\documentclass[a4paper,10pt]{article}
\usepackage{amssymb}
\usepackage{latexsym}
\usepackage{amsmath}
\usepackage{amsthm}
\usepackage{amsfonts}
\usepackage{amssymb}
\usepackage{graphicx}

\newtheorem{Th}{Theorem}
\newtheorem{Prop}{Proposition}
\newtheorem{Co}{Corollary}
\newtheorem{Lm}{Lemma}
\newtheorem{Lma}{Lemma}[section]
\newtheorem{Dfi}{Definition}
\newtheorem{Rm}{Remark}

\newcommand{\be}{\begin{equation}}
\newcommand{\ee}{\end{equation}}
\newcommand{\R}{\mathbb{R}}
\newcommand{\N}{\mathbb{N}}
\newcommand{\C}{\mathbb{C}}
\newcommand{\Z}{\mathbb{Z}}

\newcommand{\essinf}{\rm{essinf}}

\newcommand{\reset}{\setcounter{equation}{0}\setcounter{Th}{0}\setcounter{Prop}{0}\setcounter{Co}{0}
\setcounter{Lm}{0}\setcounter{Rm}{0}}
\def\La{\Lambda}
\def\La{\Lambda}
\def\ti{\tilde}
\def\lf{\left}
\def\rg{\right}

\def\al{\alpha}
\def\la{\lambda}

\def\ep{\varepsilon}
\def\ds{\displaystyle}
\def\ov{\overline}

\def\om{\omega}
\def\p{\partial}
\def\bn{\vec{n}}

\def\bbe{\vec{e}}

\def\bP{\vec{\Phi}}

\newcommand{\D}{D}

\newcommand{\Area}{Area}
\newcommand{\Sp}{\mathbb{S}}
\DeclareMathOperator{\diam}{diam}
\DeclareMathOperator{\dist}{dist}

\DeclareMathOperator{\genus}{genus}

\newcommand{\Rp}{\mathbb{R}^p}

\begin{document}
\title{Immersed Spheres of Finite Total Curvature into Manifolds.}
\author{Andrea Mondino\footnote{Scuola Normale Superiore, Piazza dei Cavalieri 7, 56126 Pisa, Italy, E-mail address: andrea.mondino@sns.it. 
A. Mondino is supported by a Post Doctoral Fellowship in the ERC grant ”Geometric Measure Theory in non Euclidean 
Spaces” directed by Prof. Ambrosio.} , Tristan Rivi\`ere\footnote{Department of Mathematics, ETH Zentrum,
CH-8093 Z\"urich, Switzerland.}}
\maketitle

{\bf Abstract :}{\it We  prove that a sequence of possibly branched, weak immersions of the two-sphere $\Sp^2$ into an arbitrary compact
riemannian manifold $(M^m,h)$ with uniformly bounded area and uniformly bounded $L^2-$norm of the second fundamental form either collapse 
to a point or weakly converges as current, modulo extraction of a subsequence, to a Lipschitz mapping of $\Sp^2$ 
and whose image is made of a connected union of finitely many, possibly branched, weak immersions of $\Sp^2$ with finite total curvature.
We prove moreover that if the sequence belongs to a class $\gamma$ of $\pi_2(M^m)$ the limiting lipschitz mapping of $\Sp^2$ realizes this class as well. }

\medskip

\noindent{\bf Math. Class.} 30C70, 58E15, 58E30, 49Q10, 53A30, 35R01, 35J35, 35J48, 35J50.

\section{Introduction}

Througout the paper $(M^m,h)$ denotes a connected riemannian manifold and for any $x_0\in M^m$ we denote respectively by $\pi_2(M^m,x_0)$ the {\it homotopy groups of based maps} form $\Sp^2$
into $M^m$ sending the south pole to $x_0$  and by $\pi_0(C^0(\Sp^2,M^m))$ the  free homotopy classes. It is well known that the group $\pi_2(M^m,x_0)$ for different $x_0$'s are
isomorphic to each other and $\pi_2(M)$ denotes any of the $\pi_2(M^m,x_0)$ modulo isomorphisms.

Following the classical approach of Douglas and Rado for the Plateau Problem,  Sacks and Uhlenbeck proceeded to the minimization of the Dirichlet energy among mappings $\vec{\Phi}$ of the two sphere $\Sp^2$ into $M^m$
\[
E(\vec{\Phi})=\frac{1}{2}\int_{\Sp^2}|d\vec{\Phi}|_h^2\ dvol_{\Sp^2}
\]
within a fixed based homotopy class in $\pi_2(M^m,x_0)$ in order to generate area minimizing, possibly branched, immersed spheres realizing this homotopy class.

Though the paper had a major impact in mathematics (analysis, geometry, differential topology...etc) as being the funding work for the concentration compactness theory, the program of Sacks and Uhlenbeck
was only partially successful. Indeed the possible loss of compactness arising in the minimization process can generate a union of immersed spheres
realizing the corresponding \underbar{free homotopy class} but for which the underlying component in the \underbar{based homotopy group} $\pi_2(M^m,x_0)$ may have been forgotten.

For instance, assuming $\pi_1(M^m)=0$ in such a way that the  free homotopy classes $\pi_0(C^0(\Sp^2,M^m))$ identify
\footnote{ We recall that the free homotopy classes is in one to one correspondence with  the set of orbits of the action of $\pi_1(M^m)$ on $\pi_2(M^m,x_0)$.
} to the homotopy classes of $\pi_2(M^m)$,
one could consider the situation where two given distinct components $\gamma_1$ and $\gamma_2$ of $C^0(\Sp^2,M^m)$ possess absolute minimizer
of the area among conformal immersion of $\Sp^2$ which are  ``far apart'' in the sense that the union of the images of two such minimizers respectively in $\gamma_1$ and in $\gamma_2$  is never connected.
Considering now the component $\gamma_1+\gamma_2$ in $\pi_2(M^m)$, a  minimizing sequence for the energy in this component is likely to converge in the class of currents to the union of two disconnected
conformal immersions of $\Sp^2$ each realizing respectively $\gamma_1$ and $\gamma_2$. Such a disconnected limit is achieved though the minimizing sequence was made of connected objects.
It is expected that, from the mapping point of view, the limit is given by the two spheres connected by a minimizing geodesic which has been lost in the current convergence process.
It is very hard in the Sacks Uhlenbeck's approach to decide for which class this phenomenon indeed occurs and to identify the  classes $\gamma_1$, $\gamma_2$ which are realized
by minimal conformal immersions and to dissociate them from the somehow {\it not satisfying classes} $\gamma_1+\gamma_2$. At least Sacks and Uhlenbeck could
prove that the set of {\it satisfying classes} is generating the free homotopy group $\pi_2(M^m)$.

\begin{Th}\cite{SaU}
\label{th-I.1}
There exists a set of free homotopy classes $\Lambda_i\in \pi_0(C^0(\Sp^2,M^m))$ such that elements $\{\la\in\La_i\}$ generate $\pi_2(M^m)$ acted on by $\pi_1(M^m)$ and each 
$\Lambda_i$ contains a conformal branched immersion of a sphere having least area among maps of $\Sp^2$ into $M^m$ which lie in $\Lambda_i$.\hfill $\Box$
\end{Th}

\medskip

In order to remedy to the difficulty to identify which class in $\pi_2(M^m)$ is realized by branched minimal immersions, in the present work  we 
generate an alternative representative when the given class is not achieved. To that aim we will consider the minimization of the following energy
\[
L(\vec{\Phi}):=\int_{\Sp^2}\lf[1+|\vec{H}_{\vec{\Phi}}|_h^2\rg]\ dvol_{g}\quad,
\]
among possibly branched immersions $\vec{\Phi}$, where $dvol_{g}$ denotes the volume form associated to the induced metric $g:=\vec{\Phi}^\ast h$ by $\vec{\Phi}$ on 
$\Sp^2$ and $\vec{H}_{\vec{\Phi}}$ is the mean curvature vector associated to the immersion.

\medskip

Observe that this energy we propose to minimize is the sum of the area and the so called {\it Willmore energy} of the immersion :
\[
L(\vec{\Phi}):=A(\vec{\Phi})+W(\vec{\Phi})\quad.
\]
Minimizing this energy is a natural generalization of Sacks Uhlenbeck's procedure in the sense that, if a class  $\gamma$ in $\pi_2(M^m)$ possesses an area minimizing
immersion $\vec{\Phi}$ then $\vec{H}_{\vec{\Phi}}\equiv 0$ and thus $\vec{\Phi}$ has also to minimize $L$ in its homotopy class. Moreover, as we will see below in Theorem~\ref{th-I.2}
and in the subsequent work \cite{MoRi}, the minimization procedure of $L$ has the advantage to always provide a smooth
representative of any chosen class of $\pi_2(M^m)$.

\medskip

Minimizing $L$ among smooth immersions is of course a-priori an ill posed variational problem exactly like minimizing the Dirichlet energy $E$ right away
among $C^1$ maps has little chance of success. In \cite{Ri3} (see also \cite{Ri1}), the second author introduced a suitable setting for dealing with minimization problems
whose highest order term is given by the Willmore energy. We now recall the notion of {\it weak branched immersions with finite total curvature}.

\medskip

By virtue of Nash theorem we can always assume that $M^m$ is isometrically embedded in some euclidian space ${\R}^n$. We first define the Sobolev spaces
from $\Sp^2$ into $M^m$ as follows, for any $k\in {\N}$ and $1\le p\le\infty$
\[
W^{k,p}(\Sp^2,M^m):=\lf\{u\in W^{k,p}(\Sp^2,{\R}^n)\ \mbox{ s. t. }\ u(x)\in M^m\ \mbox{ for a.e. }x\in \Sp^2\rg\}\quad.
\]
We introduce the space of {\it possibly branched lipschitz immersions} : A map $\vec{\Phi}\in W^{1,\infty}(\Sp^2,M^m)$ is a 
{\it possibly branched lipschitz immersion} if 
\begin{itemize}
\item[i)] there exists $C>1$ such that, for a.e. $x\in \Sp^2$,
\be
\label{I.1}
 C^{-1}|d\vec{\Phi}|^2(x)\le |d\vec{\Phi}\wedge d\vec{\Phi}|(x)\le |d\vec{\Phi}|^2(x) \quad,
\ee
where the norms of the different tensors have been taken with respect to the standard metric on $\Sp^2$ and with respect to the metric $h$ on $M^m$ and
where $d\vec{\Phi}\wedge d\vec{\Phi}$ is the tensor given in local coordinates on $\Sp^2$ by $$d\vec{\Phi}\wedge d\vec{\Phi}:=2\ \p_{x_1}\vec{\Phi}\wedge\p_{x_2}\vec{\Phi}\ dx_1\wedge dx_2\in \wedge^2T^\ast \Sp^2\otimes \wedge^2 T_{\vec{\Phi}(x)}M^m \quad.$$
\item[ii)] There exists at most finitely many points $\{a_1\cdots a_N\}$ such that for any compact $K\subset \Sp^2\setminus \{a_1\cdots a_N\}$
\be
\label{I.2}
\essinf_{x\in K}|d\vec{\Phi}|(x)>0 \quad .
\ee
\end{itemize}
For any {\it possibly branched lipschitz immersion} we can define almost everywhere the {\it Gauss map} 
\[
\vec{n}_{\vec{\Phi}}:=\star_h\frac{\p_{x_1}\vec{\Phi}\wedge\p_{x_2}\vec{\Phi}}{|\p_{x_1}\vec{\Phi}\wedge\p_{x_2}\vec{\Phi}|}\ \in\ \wedge^{m-2} T_{\vec{\Phi}(x)}M^m \quad,
\]
where $(x_1,x_2)$ is a local arbitrary choice of coordinates on $\Sp^2$ and $\star_h$ is the standard Hodge operator associated to the metric  $h$ on multi-vectors in $TM$.

With these notations we introduce the following definition.
\begin{Dfi}
\label{df-I.1}
A lipschitz map $\vec{\Phi}\in  W^{1,\infty}(\Sp^2,M^m)$ is called ``weak, possibly branched, immersion'' if $\vec{\Phi}$ satisfies
(\ref{I.1}) for some $C\ge 1$, if it satisfies (\ref{I.2}) and if the Gauss map satisfies
\be
\label{I.3}
\int_{\Sp^2}|D\vec{n}_{\vec{\Phi}}|^2\ dvol_g<+\infty\quad,
\ee
where $dvol_g$ is the volume form associated to $g:=\vec{\Phi}^\ast h$ the pull-back metric of $h$ by $\vec{\Phi}$ on $\Sp^2$, $D$ denotes the covariant derivative with respect to $h$ and the norm $|D\vec{n}_{\vec{\Phi}}|$ of the tensor $D\vec{n}_{\vec{\Phi}}$ is taken with respect
to $g$ on $T^\ast \Sp^2$ and $h$ on $\wedge^{m-2}TM$.
The space of ``weak, possibly branched, immersions'' of $\Sp^2$ into $M^m$ is denoted by ${\mathcal F}_{\Sp^2}$. \hfill $\Box$
\end{Dfi}
For $\vec{\Phi}\in{\mathcal F}_{\Sp^2}$ we denote 
\[
F(\vec{\Phi})=\frac{1}{2}\int_{\Sp^2}|D\vec{n}_{\vec{\Phi}}|^2\ dvol_{g}
\]
and
\[
G(\vec{\Phi}):=A(\vec{\Phi})+F(\vec{\Phi})=\int_{\Sp^2}\lf[1+\frac{|D\vec{n}_{\vec{\Phi}}|^2}{2}\rg]\ dvol_{g}\quad.
\]
Using M\"uller-Sv\v{e}r\'ak theory of weak isothermic charts (see \cite{MS}) and H\'elein's moving frame technique  (see \cite{He}) one can prove the following proposition (see \cite{Ri1})
\begin{Prop}
\label{pr-I.1}
Let $\vec{\Phi}$ be a  weak, possibly branched, immersion of $\Sp^2$ into $M^m$ in ${\mathcal F}_{\Sp^2}$. Then there exists a bilipschitz homeomorphism $\Psi$ of $\Sp^2$
such that $\vec{\Phi}\circ\Psi$ is weakly conformal : it satisfies almost everywhere on $\Sp^2$
\[
\lf\{
\begin{array}{l}
\ds |\p_{x_1}(\vec{\Phi}\circ\Psi)|_h^2=|\p_{x_2}(\vec{\Phi}\circ\Psi)|_h^2\quad\\[5mm]
\ds h(\p_{x_1}(\vec{\Phi}\circ\Psi),\p_{x_2}(\vec{\Phi}\circ\Psi))=0 \quad,
\end{array}
\rg.
\]
where $(x_1,x_2)$ are local arbitrary conformal coordinates in $\Sp^2$ for the standard metric. Moreover $\vec{\Phi}\circ\Psi$ is in $W^{2,2}\cap W^{1,\infty}(\Sp^2,M^m)$.\hfill $\Box$
\end{Prop}

\medskip

\begin{Rm}
\label{rm-I.1}
In view of  Proposition~\ref{pr-I.1}, a careful reader could wonder why we do not work  with conformal $W^{2,2}$ weak, possibly branched, immersions only and why we
do not impose for the membership in ${\mathcal F}_{\Sp^2}$, $\vec{\Phi}$ to be conformal from the beginning. The reason why this would be a wrong strategy and why
we have to keep the flexibility
for weak immersions not to be necessarily conformal  will appear clearly
in the second paper \cite{MoRi} while studying the variations of $L$ and while considering general perturbations which do not have to respect
infinitesimally the conformal condition. \hfill $\Box$ 
\end{Rm}

\medskip

\medskip

In the sequel we shall denote by ${\mathcal M}^+(\Sp^2)$ the non-compact M\"obius group of positive conformal diffeomorphisms of the 2-sphere $\Sp^2$.
Our main result in the present work is the following weak-semi-closure result of ${\mathcal F}_{\Sp^2}$. 

\begin{Th}
\label{th-I.2}
Let $\vec{\Phi}_k$ be a sequence of  weak, possibly branched, conformal immersions of ${\mathcal F}_{\Sp^2}$ such that
\be
\label{I.4}
\limsup_{k\rightarrow +\infty}\int_{\Sp^2}\lf[1+|D\vec{n}_{\vec{\Phi}_k}|_h^2\rg]\ dvol_{g_k}<+\infty\quad,
\ee
where $dvol_{g_k}$ denotes the volume form associated to the induced metric $g_k:=\vec{\Phi}_k^\ast h$ by $\vec{\Phi}_k$ on 
$\Sp^2$ and $D\vec{n}_{\vec{\Phi}_k}$ is the covariant derivative in $(M,h)$ of the normal space $\vec{n}_{\vec{\Phi}_k}$ to $\vec{\Phi}_k$. Assume moreover that
\be
\label{I.4a}
\liminf_k\, diam(\vec{\Phi}_k(\Sp^2))>0\quad.
\ee
 Then there exists
a subsequence that we still denote $\vec{\Phi}_k$, there exists a family of bilipschitz homeomorphisms $\Psi_k$, there exists 
a finite family of sequences $(f^i_k)_{i=1\cdots N}$ of elements in ${\mathcal M}^+(\Sp^2)$,
there exists a finite family of natural integers $(N^i)_{i=1\cdots N}$ and for each $i\in\{1\cdots N\}$ there exists finitely many points of $\Sp^2$, $b^{i,1}\cdots b^{i,N^i}$ such that
\be
\label{I.5}
\vec{\Phi}_k\circ\Psi_k\longrightarrow \vec{f}_\infty\quad\mbox{ strongly in }C^0(\Sp^2,M^m)\quad,
\ee
where $\vec{f}_\infty\in W^{1,\infty}(\Sp^2,M^m)$, moreover
\be
\label{I.6}
\vec{\Phi}_k\circ f^i_k\rightharpoonup \vec{\xi}_\infty^i\quad\mbox{ weakly in } W^{2,2}_{loc}(\Sp^2\setminus\{b^{i,1}\cdots b_{i,N^i}\})\quad,
\ee
where $\vec{\xi}_\infty^j\in{\mathcal F}_{\Sp^2}$ is conformal.  In addition we have 
\be
\label{I.7}
\ds\vec{f}_\infty(\Sp^2)=\bigcup_{i=1}^N\vec{\xi}_\infty^i(\Sp^2)\quad,
\ee
moreover
\be
\label{I.6b}
Area(\vec{\Phi}_k)=\int_{\Sp^2}1\ dvol_{g_{\vec{\Phi}_k}}\longrightarrow Area(\vec{f}_\infty)=\sum_{i=1}^N Area(\vec{\xi}_\infty^i)\quad,
\ee
and finally
\be
\label{I.8}
(\vec{f}_\infty)_\ast[\Sp^2]=\sum_{i=1}^N(\vec{\xi}_\infty^i)_\ast[\Sp^2]\quad,
\ee
where for any Lipschitz mapping $\vec{a}$ from $\Sp^2$ into $M^m$, $(\vec{a})_\ast[\Sp^2]$ denotes the current
given by the push-forward by $\vec{a}$ of the current of integration over $\Sp^2$ : for any smooth two-form
$\om$ on $M^m$
\[
\lf<(\vec{a})_\ast[\Sp^2],\om\rg>:=\int_{\Sp^2}(\vec{a})^\ast\om\quad.
\] 
\hfill $\Box$
\end{Th}

In \cite{MoRi}, for any given homotopy class $\gamma$ of $\pi_2(M^m,x_0)$, we will construct such an $\vec{f}_\infty$ representing $\gamma$
and for which the associated $\vec{\xi}_\infty^j$ are 
smooth, possibly branched, conformal immersions satisfying the {\it area constrained Willmore equation}.
 We have then succeeded in realizing each \underbar{based} homotopy class with a smooth  mapping of $\Sp^2$ whose image equals
 a connected union of smooth {\it area constrained Willmore}, possibly branched, conformal immersions of $\Sp^2$.
 
 \medskip
 
 The paper is organized as follows : in section 2 we establish uniform controls of the number of branched points as well
 as a uniform control of the $L^{2,\infty}-$gradient of the conformal factor. In section 3 we establish a concentration compactness result for weak branched conformal immersions
 from an arbitrary riemann surface into a closed riemannian manifold. In section 4 we present a way to ``normalize'' the parametrization of sequences of conformal
 weak immersions with uniformly bounded $L$ energy and non shrinking diameter in order to converges to a non trivial conformal immersion of $\Sp^2$. In section 5 we develop
 a procedure in order to decompose $\Sp^2$ into subdomains which, modulo renormalization, will converge to non trivial conformal immersions of $\Sp^2$ and which exhaust
 completely the image. In section 6 we prove the main theorem~\ref{th-I.2} and finally in the last section we will prove the extension 
 of the weak closure theorem~\ref{th-I.2} to bubble trees of elements in ${\mathcal F}_{\Sp^2}$. This last result will be the starting point to
 the proof of the realization of homotopy classes by bubble trees of Willmore spheres in \cite{MoRi}.
 
 \medskip
 
 \noindent {\bf Acknowledgments.} This work was written while the first author was visiting the {\it Forschungsinstitut f\"ur Mathematik} at the ETH Z\"urich. He would like
 to thank the Institut for the hospitality and the excellent working conditions.

\section{Branched points control and $L^{2,\infty}-$estimates of the gradient of the conformal factor in conformal parametrization.}
\reset 
Let $(M^m,h)$ be an $m$-dimensional Riemannian manifold and  $(\Sigma, c_o)$ a smooth, closed Riemann surface; without loss of generality we can assume that $(\Sigma, c_o)$ has a metric with constant curvature and unit area (see for example \cite{Jo}). First of all we want to define the set of  branched conformal immersions with finite total curvature of $(\Sigma, c_0)$ into $(M^m,h)$.
Consider a map $\vec{\Phi}\in W^{1,\infty}(\Sigma, (M,h))$ and consider the following conditions:

1) Conformality: almost everywhere on $\Sigma$ 
\be\label{Cond:Conf}
\lf\{
\begin{array}{l}
|\p_{x_1} \vec{\Phi}|^2_h=|\p_{x_2} \vec{\Phi}|^2_h\\[5mm]
h(\p_{x_1} \vec{\Phi},\p_{x_2} \vec{\Phi})=0 \quad .
\end{array}
\rg.
\ee

2) Finite number of singular points: There exists at most finitely many points $\{b_1\cdots b_{N_{\bP}}\}$ such that for any compact $K\subset \Sp^2\setminus \{b_1\cdots b_{N_{\bP}}\}$
\be
\label{Cond:FNSP}
\essinf_{x\in K}|d\vec{\Phi}|(x)>0 \quad .
\ee
Observe that on $\Sigma\setminus\{b_1,\ldots,b_{N_{\bP}}\}$ we can define the normal space $\bn\in \Gamma(\vec{\Phi}^{-1}(TM))$
as
\be\label{eq:defn}
\bn_{\bP}:=\star_h \frac{\p_{x_1} \vec{\Phi} \wedge \p_{x_2} \vec{\Phi}} {|\p_{x_1} \vec{\Phi} \wedge \p_{x_2} \vec{\Phi}|},
\ee  
so $\bn$ is defined almost everywhere and is an $L^\infty$ vector field on the whole $\Sigma$.

3) Finite total curvature, that is we ask that $\vec{n}$ is a $W^{1,2}$  $m-2$-vector field
\be\label{Cond:FTC}
\int_{\Sigma} |\D \bn_{\bP}|^2 dx < \infty.
\ee
Notice that by the invariance of the integrand with respect to conformal changes of metric on the surface $\Sigma$, it doesn't matter which representant in the conformal class we choose to compute \eqref{Cond:FTC}.

A map $\vec{\Phi}\in W^{1,\infty}((\Sigma,c_0), (M,h))$ which satisfies \eqref{Cond:Conf}, \eqref{Cond:FNSP}, \eqref{Cond:FTC} is called \emph{weak branched conformal immersion with finite total curvature} and we denote with
\be\label{def:FcS}
{\cal F}^{c_0}_{\Sigma}:= \{\vec{\Phi}\in W^{1,\infty}((\Sigma,c_0), (M,h)) \text{ satisfying  \eqref{Cond:Conf}, \eqref{Cond:FNSP} and \eqref{Cond:FTC} }\}.  
\ee

\subsection{Behaviour at singular points}\label{SubSec:SingPoints}

First of all let us recall the following lemma proved by the second author (see Lemma A.5 in \cite{Ri3}. This lemma was originally proved in \cite{MS} using quite different arguments.
We used instead in \cite{Ri3} the {\it  moving frame approach} of F. H\'elein - see \cite{He}). Before stating it let us recall that $Gr_{p-2}(\R^p)$ denotes the Grassmannian of $p-2$-dimensional linear subspaces of $\R^p$.

\begin{Lm}\label{lem:branched}
Let $\vec{\Phi}$ be a conformal immersion of $\D^2\setminus\{0\}$ into $\R^p$ in $W^{2,2}_{loc}(\D^2\setminus \{0\}, \R^p)$ and such that the conformal factor $\log |\nabla \vec{\Phi}| \in L^{\infty}_{loc}(\D^2\setminus \{0\})$. Assume $\bP$ extends to a map in $W^{1,2}(\D^2)$ and that the corresponding Gauss map $\bn_{\bP}$ also extends to a map in $W^{1,2}(\D^2, Gr_{p-2}(\R^p))$. Then $\bP$ realizes a lipschitz conformal immersion of the whole disc $\D^2$ and there exists a positive integer $n$ and a constant $C$ such that
\be\label{eq:branched}
(C-o(1)) |z|^{n-1} \leq \lf|\frac{\p \bP}{\p z} \rg| \leq (C+o(1)) |z|^{n-1}.
\ee
\end{Lm}

Observe that if $\bP$ is a weak branched conformal immersion with finite total curvature into the Riemannian manifold $(M^m,h)$ then by Nash embedding theorem we can see the ambient manifold $(M^m,h)$ isometrically embedded in some $\R^n$ and $\bP$ as a weak branched conformal immersion with finite total curvature into $\R^n$ taking values in the submanifold $M$. 
\medskip

By well established estimates on the conformal factor (see for example the notes of Rivi\`ere \cite{Ri1} pag. 120-136) we have that $\log|\nabla\bP| \in W^{1,2}_{loc}\cap L^\infty_{loc} (\D^2\setminus\{0\})$, moreover by assumption we have that $\bn_{\bP}\in W^{1,2}(\D^2)$; both the informations together imply that $\bP \in W^{2,2}_{loc}(\D^2\setminus\{0\})$. By assumption we have that $\bP \in W^{1,\infty}(\D^2)$ so we are in the assuptions of Lemma \ref{lem:branched} and we can conclude that the singular points $b_1,\ldots,b_{N_{\bP}}$ are actually branch points with \emph{positive} branching order (i.e. $n\geq 2$) given by \eqref{eq:branched}. 
\medskip

From the discussion above we know the behaviour of the conformal factor $\lambda=\log \lf(\lf|\frac{\p \bP}{\p z} \rg|\rg)$ near the branch points $b_1, \ldots, b_{N_{\bP}}$; taking conformal coordinates parametrizing  a punctured neighboorhod of $b_i$ on the puctured disc $\D^2 \backslash \{0\}$ we observe that the following conditions are satisfied:

\be\label{syst:ConfFacBP}
\lf\{
\begin{array}{l}
-\Delta_0 \lambda = K_{\bP} e^{2\lambda}-K_0 \quad \text{on } \Sigma\backslash\{b_1,\ldots,b_{N_{\bP}}\} \\[5mm]
(n_j-1) \log |z|+ (C-o(1)) \leq \lambda(z) \leq (n_j-1) \log |z|+ (C+o(1))\quad,
\end{array}
\rg.
\ee
where of course $K_{\bP}$ is the Gauss curvature of the metric $g=(\bP)^*(h)$ given by the immersion, $K_0$ is the (constant) Gauss curvature of the reference metric of $(\Sigma,c_0)$, $\Delta_0$ is the Laplace Beltrami operator on $(\Sigma,c_0)$ and $n_i \in \N$ are given by Lemma \ref{lem:branched} applied to a neighboorod of $b_i$; obseve that the first equation is just the Liouville equation on the regular part of the immersion $\bP$.
\medskip

Since $\lambda \in L^1_{loc}(\Sigma)$ then $-\Delta_0 \lambda$ is a distribution which, using the first equation of \eqref{syst:ConfFacBP}, has singular support contained in $\{b_1,\ldots,b_{N_{\bP}}\}$. By Schwartz Lemma it follows that 
$-\Delta_0 \lambda-K_{\bP} e^{2\lambda}+K_0$ is a finite sum of deltas and derivatives of deltas at the points $b_i$. Using the second condition of \eqref{syst:ConfFacBP} we get that $\lambda$ satisfies the following singular PDE on the whole $\Sigma$ (for more details see the Appendix of \cite{BDM}; see also \cite{McOw} and \cite{Tro}):

\be\label{eq:PDElamb}
-\Delta_0 \lambda = K_{\bP} e^{2\lambda}-K_0-2\pi \sum_{j=1} ^{N_{\bP}} \lf[(n_j-1) \delta_{b_j}\rg] 
\ee
where of course $\delta_{b_j}$ is the delta centred at the point $b_j$.

\subsection{Estimates on the gradient of the conformal factor and on the sum of the branching orders}

\begin{Lm}\label{lem:EstCFBranch}
Let $\bP\in {\cal F}^{c_0}_\Sigma$ be a weak branched conformal immersion of the Riemann surface $(\Sigma, c_0)$ into the Riemannian manifold $(M^m,h)$. Called $b_1,\ldots,b_{N_{\bP}}$ the branch points and $n_1\geq 1,\ldots,n_{N_{\bP}}\geq 1$ the corresponding branching orders given by the previous discussion in Subsection \ref{SubSec:SingPoints}, we have the following estimates on the sum of the branching orders and on the $L^{2,\infty}$ norm of the conformal factor:
\be \label{eq:EstBranch1}
 \sum_{j=1}^{N_{\bP}} (n_j-1) \leq   \frac {1} {4\pi} \int_\Sigma |\D \bn_{\bP}|^2 \, dvol_g+\frac {1} {2\pi} \lf(\sup_{\bP(\Sigma)} |\bar{K}|\rg) \Area_g(\Sigma)-\chi_E(\Sigma),
\ee

\be\label{eq:estNablalam1}
\|\nabla \lambda\|_{L^{2,\infty}(\Sigma)} \leq C_{(\Sigma,c_0)} \lf[ 1+ \int_\Sigma |\D \bn_{\bP}|^2 \, dvol_g + \lf(\sup_{\bP(\Sigma)}|\bar{K}|\rg) \Area_g(\Sigma)\rg], 
\ee
where  $\Area_g$ is the area of $\Sigma$ with respect to the metric $g:=\bP^*h$, $\bar{K}$ is the sectional curvature of $(M^m,h)$, $\chi_E(\Sigma)$ is the Euler characteristic of $\Sigma$, $C_{(\Sigma,c_0)}$ is a constant depending only on the Riemann surface $(\Sigma,c_0)$, and the $L^{2,\infty}$ norm on $\Sigma$ is taken with respect to the metric of scalar constant 
curvature and volume 1 on $(\Sigma,c_0)$.
\hfill $\Box$
\end{Lm}

\begin{proof}
By the Gauss equation (which still holds almost everywhere on a weak branched  conformal immersion with finite total curvature by an approximation argument, in the smooth case see for example \cite{DoC} pag.130) we have
\be\label{eq:Gauss}
K_{\bP}=\bar{K}(\bP _*(T\Sigma))+\lf< {\mathbb I}_{11}, {\mathbb I}_{22} \rg>- |{\mathbb I}_{12}|^2
\ee
where $\bar{K}(\bP _*(T\Sigma))$ is the sectional curvature of $(M,h)$ computed on  the plane $\bP _*(T\Sigma)$, $K_{\bP}$ is the Gauss curvature of $(\Sigma,g)$ and $\mathbb I_{ij}$ is the second fundamental form which is defined almost everywhere in the usual sense:
$${\mathbb I}_{ij}:= -\sum_{\alpha=1}^{m-2}\lf< \D_{\bbe _i} \bn_\alpha, \bbe_j \rg> \bn_{\alpha} $$
where $\bbe_i(q):=\frac{1}{|\p _{x_i} \bP|}\p _{x_i} \bP|_q$ is an orthonormal frame of $\bP_*(T_q\Sigma)$ and $(\bn_1(q),\ldots,\bn_{m-2}(q))$ realizes a positive orthonormal basis of $(\bP_*(T_q\Sigma))^{\perp}$ so that $\bn_{\bP}=\bn_1\wedge\ldots\wedge \bn_{m-2}$. Since $|{\mathbb I}|^2=|\D \bn_{\bP}|^2$, by Schwartz's inequality we have the estimate
\be\label{eq:EstK}
\begin{array}{l}
|K_{\bP}|\leq |\bar{K}(\bP _*(T\Sigma))| + |{\mathbb I}_{11}| |{\mathbb I}_{22}| + |{\mathbb I}_{12}|^2\\[5mm]
\quad\quad\leq |\bar{K}(\bP _*(T\Sigma))| + \frac{1}{2}\lf(|{\mathbb I}_{11}|^2+ |{\mathbb I}_{22}|^2 +2 |{\mathbb I}_{12}|^2 \rg)=|\bar{K}(\bP _*(T\Sigma))| + \frac{1}{2} |\D \bn_{\bP}|^2.
\end{array}
\ee
Integrating equation \eqref{eq:PDElamb} and using the estimate \eqref{eq:EstK} we get the following estimate on the number of branch points in terms of the total curvature and the area:
\be \label{eq:EstBranch}
2\pi \sum_{j=1}^{N_{\bP}} (n_j-1)=\int_{\Sigma} (K_{\bP} e^{2\lambda}) \, dvol_{c_0}-K_0   \leq  \frac 1 2 \int_\Sigma |\D \bn_{\bP}|^2 \,dvol_{g}+\lf(\sup_{\bP(\Sigma)} |\bar{K}|\rg) \Area_g(\Sigma)-2\pi \chi_E(\Sigma)\quad,
\ee
where $dvol_{c_0}$ is the volume form with respect to constant scalar curvature metric of volume $1$ on $(\Sigma, c_0)$.
Now putting together equation \eqref{eq:PDElamb} and estimate \eqref{eq:EstBranch} we get the following bound on the norm of $-\Delta_0 \lambda$ as Radon measure:
\be
\|\Delta_0 \lambda\|_{{\cal M}(\Sigma)}\leq 4\pi |\chi_E(\Sigma)|+ \int_\Sigma |\D \bn_{\bP}|^2\,dvol_{g} +2 \lf(\sup_{\bP(\Sigma)} |\bar{K}|\rg) \Area_g(\Sigma)
\ee
which implies that there exists a constant $C=C_{(\Sigma,c_0)}$ depending just on the Riemann surface $(\Sigma,c_0)$ such that (see \cite{He} pag. 138)
\be\label{eq:estNablalam}
\|\nabla \lambda\|_{L^{2,\infty}(\Sigma)} \leq C_{(\Sigma,c_0)} \lf[ 1+ \int_\Sigma |\D \bn_{\bP}|^2\,dvol_{g}+ \lf(\sup_{\bP(\Sigma)}|\bar{K}|\rg) \Area_g(\Sigma)\rg] 
\ee
\end{proof}

\section{Concentration compactness for weak branched conformal immersions of general Riemann surfaces into manifolds}
\reset
\begin{Prop}\label{Prop:ConcComp}
Let $(\Sigma,c_k)$ be a sequence of Riemann surfaces defined on the same topological surface $\Sigma$, consider weak branched conformal immersions ${\cal F}_{\Sigma}^{c_k}\ni\bP_k\hookrightarrow (M^m,h)$ into the $m$-dimensional Riemannian manifold $(M^m,h)$ and assume the following conditions hold true:
\begin{itemize}
\item[{\rm (i)}] the conformal structures $c_k$  are contained in a fixed compact subset of the moduli space of $\Sigma$.
\item[{\rm (ii)}] the areas of $\Sigma$ induced by the immersions are uniformly bounded from above: called $g_k:=\bP_k ^*(h)$,
$$\Area_{g_k}(\Sigma)\leq C;$$
\item[{\rm (iii)}] the energies of the $\bP_k$ are uniformly bounded from above:
$$F(\bP_k)=\frac{1}{2} \int_{\Sigma} |\D \bn_{\bP_k}|^2 dvol_{g_k}= \frac{1}{2} \int_{\Sigma} | {\mathbb I}_{\bP_k}|^2 dvol_{g_k} \leq C.$$ 
\end{itemize}
Then there exists a finite set of points $\{a_1, \ldots, a_N\}\subset \Sigma$ such that, called $\lambda_k:=\log|\partial_{x_1} \bP_k|_h=\log|\partial_{x_2} \bP_k|_h$ the conformal factor of the immersion $\bP_k$, for any compact subset $K\subset \subset  \Sigma \setminus\{a_1,\ldots,a_N\}$ the following properties hold:
\begin{itemize}
\item[{\rm (a)}] there exists a constant $C_K$ depending on $K$ and on the bounds on areas and energies of $\bP_k$ such that, up to subsequences on $k$,   
$$\|e^{\lambda_k}\|_{L^\infty(K)}\leq C_K \quad ;$$ 
\item[{\rm (b)}]  either there exists a constant $C_K$ depending as above such that, up to subsequences on $k$,
 $$\|\lambda_k\|_{L^\infty(K)}\leq C_K$$
 or, up to subsequences on $k$,  
 $$\lambda_k \to -\infty \text { uniformly on } K.  $$
\end{itemize}

\end{Prop}

\begin{proof}
Consider a Nash isometric embedding $\vec{I}:(M^m,h)\hookrightarrow {\R}^n$; in this way $ \vec{I} \circ \bP_k:(\Sigma,c_k)\hookrightarrow (M^m,h) \hookrightarrow {\R}^n$ are weak branched conformal immersions in ${\R}^n$. Observe that by the compactness of $M$ and the area bound on $\bP_k$, the energies of $\vec{I}\circ \bP_k$ are also uniformly bounded:
\be\label{eq:bdd2ffRp1}
\int_{\Sigma} |d \bn_{\vec{I}\circ \bP_k}|_{g_k}^2 dvol_{g_k}\leq C.
\ee
Indeed, we have 
\be
\label{xy1}
|d \bn_{\vec{I}\circ \bP_k}|_{g_k}^2\leq |D^h \bn_{\bP_k}|_{g_k}^2+ |d \bn_{\vec{I}}|_{g_k}^2,
\ee
where $\vec{n}_{\vec{I}}$ is the Gauss map of $M^m$ in ${\R}^n$. This inequality comes simply from the fact that 
\[
 \bn_{\vec{I}\circ \bP_k}=\bn_{\bP_k}\wedge\bn_{\vec{I}}\quad.
\]
In order to get the claim it suffices to integrate on $\Sigma$ with respect to $dvol_{g_k}$ and recall the assumed bounds on area and energy of $\bP_k$. Now, since the intengrand in \eqref{eq:bdd2ffRp1} is invariant under conformal change of metric in the domain we have
\be\label{eq:bdd2ffRp}
\int_{\Sigma} |d \bn_{\vec{I}\circ \bP_k}|_{h_k}^2 dvol_{h_k}=
\int_{\Sigma} |d \bn_{\vec{I}\circ \bP_k}|_{g_k}^2 dvol_{g_k}\leq C,
\ee
where $h_k$ is the smooth reference metric on $(\Sigma,c_k)$ of constant curvature and unit volume. Since $c_k$ by assumption are contained in a compact subset of the moduli space of $\Sigma$, up to subsequences, $h_k$ strongly converges  in $C^{r}(\Sigma)$ for every $r$ to $h_\infty$, the constant curvature metric of volume 1 associated to $c_\infty$ (the limiting conformal structure). Recall that $\bP_k$ are weak branched conformal immersions with branch points $\{b^1_k,\ldots,b^{N_k}_k\}$; from the area and energy bounds, Lemma \ref{lem:EstCFBranch} implies that the number of branch points $N_k$ is uniformly bounded, so up to subsequences we can assume $N_k$ constant, say $N_k=N_1$.  For each  $k\in \N$ and  $x \in \Sigma$ we assign $\rho_x^k>0$ such that
\be
\int_{B_{\rho_x^k}(x)} |d \bn_{\vec{I}\circ \bP_k}|_{h_k}^2 dvol_{h_k}=\frac{8 \pi}{3},
\ee 
where $B_{\rho_x^k}(x)$ is the geodesic ball in $(\Sigma,h_k)$ of center $x$ and radius  $\rho_x^k$ (notice that for any closed surface $\Sigma$ immersed in $\R^n$ one has $\int_\Sigma|d \bn|^2 \, dvol_g \geq 8\pi$ so the minimal radius $\rho_x^k$ exists finite). From the covering $\{B_{\rho_x^k}(x)\}_{x\in \Sigma}$ we extract a finite Besicovich covering: every point of $\Sigma$ is covered by at most $\xi=\xi(\Sigma,h_\infty)\in \N$ balls. Let $\{B_{\rho^i_k(x^i_k)}\}_{i \in I}$ be such a cover and observe that up to subsequences: $I$ is independent on $k$, $x^i_k\to x^i_{\infty}$, $\rho^i_k\to \rho^i_{\infty}$, $b^i_k\to b^i_{\infty}$. Let 
$$J:=\{i \in I \text{ such that } \rho^i_{\infty}=0\} \quad \text{and} \quad I_0:=I\setminus J.$$
The union of the closed balls $\cup_{i\in I_0} \bar{B}_{\rho^i_{\infty}}(x^i_{\infty})$ covers $\Sigma$. Because of the strict convexity of the balls with respect to the euclidean distance ($\Sigma$=Torus, or $\Sigma=\Sp^2$ via stereographic projection) or the hyperbolic distance ($\genus(\Sigma)>1$) the points of $\Sigma$ which are not contained in the union of the \emph{open} balls $\cup_{i\in I} B_{\rho^i_\infty}(x^i_{\infty})$ cannot accumulate and therefore are isolated and hence finite (this argument was  the same  in the unbranched situation, see \cite{Ri3}). Denote 
$$\{d^1,\ldots, d^{N_2}\}:=\Sigma\setminus \cup_{i\in I_0}B_{\rho^i_{\infty}}(x^i_{\infty}).$$
Now let 
$$\{a^1,\ldots,a^N\}:=\{d^1,\ldots, d^{N_2}\}\cup \{b^1_{\infty},\ldots, b^{N_1}_{\infty}\}$$
and fix a compact subset $K\subset \subset \Sigma\backslash \{a^1,\ldots,a^N\}$. Clearly there exists $\delta>0$ such that
$K \subset \Sigma\setminus \cup_{i=1}^N \bar{B}_{\delta}(a^i)$. Since $\Sigma \setminus \cup_{i=1}^N B_{\delta}(a^i)\subset \cup_{i\in I_0} B_{\rho^i_{\infty}}(x^i_{\infty})$, there exist $0<r^i<\rho^i_{\infty}$ such that
\be\label{3.20}
\Sigma\setminus \cup_{i=1}^N B_{\delta}(a^i) \subset \cup_{i\in I_0} B_{r^i}(x^i_{\infty}) 
\ee
and for $k$ large enough one has, for any $i \in I_0$, $B_{r^i}(x^i_\infty)\subset B_{\rho^i_k}(x^i_k)$. Let $s^i=(r^i+\rho^i_\infty)/2$. Again, for $k$ large enough, $B_{s^i}(x^i_\infty)\subset B_{\rho^i_k}(x^i_k)$ for $i\in I_0$. Recall that by construction we have the crucial estimate
\be \label{eq:3.21}
\int_{B_{s^i}(x^i_{\infty})} |d \bn_{\vec{I}\circ \bP_k}|_{h_k}^2 dvol_{h_k}\leq\frac{8 \pi}{3}.
\ee
Now we claim that for every $i$ and, up to subsequences, $k$ there exists a constant $\bar{\lambda}^i_k$ such that for every radius $r^i < r < s^i$:
\be\label{eq:3.22}
\|\lambda_k-\bar{\lambda}^i_k\|_{B_r(x^i_\infty)\setminus \cup_{j=1}^{N_1} B_{\delta/2}(b^j_{\infty})}\leq C_{r,\delta}.
\ee
In order to prove the claim observe that for each ball $B_{s^i}(x^i_{\infty}), i \in I_0$, we have two possibilities: either it contains a limit of branch points $b^{j_i}_{\infty}$ or not. 
In the second case, since \eqref{eq:3.21} holds,  we can apply the construction of moving frames (pag.23, 49 \cite{Ri3}, pag.139-142 \cite{Ri1}) and on the slightly smaller  ball $B_r(x^i_\infty)$  we have an $L^\infty$ control of the conformal factor $\lambda_k$ up to a constant $\bar{\lambda}^i_k$ as desired.
In the first case $B_{s^i}(x^i_{\infty})$ contains some  limit of branch points. 
Consider a finite cover of 
 $$B_{s^i}(x^i_{\infty})\setminus \cup_{j=1}^{N_1}B_{\delta/4}(b^j_\infty)$$
  of balls contained in $ B_{\rho^i_k}(x^i_k) \setminus \cup_{j=1}^{N_1}B_{\delta/8}(b^j_\infty)$ ( recall that, for large $k$,   $B_{s^i}(x^i_{\infty})\subset B_{\rho^i_k}(x^i_k)$ with boundaries at strictly positive distance); 
  so for each ball $\tilde{B}$ in this last covering we have $\int_{\tilde{B}} |d \bn_{\vec{I}\circ \bP_k}|_{h_k}^2 dvol_{h_k}\leq\frac{8 \pi}{3}$ for large $k$. Hence on each ball in the covering we costruct Helein's moving frames as before getting $L^\infty$ control of $\lambda_k$ on any slightly smaller ball. This gives the estimate \eqref{eq:3.22} on  $B_{r}(x^i_{\infty})\setminus \cup_{j=1}^{N_1}B_{\delta/2}(b^j_\infty) \subset B_{s^i}(x^i_{\infty})\setminus \cup_{j=1}^{N_1}B_{\delta/4}(b^j_\infty)$.

Now, from the area bound, $\bar \lambda^i_k\leq C$ independent on $k$ and $i$. Indeed if it is not the case there would exist $i_0$ and a subsequence in $k$ such that $\bar \lambda^{i_0}_k\to + \infty$; therefore by \eqref{eq:3.22}
$$
\begin{array}{l}
\ds\Area_{g_k}(B_r(x^{i_0}_\infty)\setminus \cup_{j=1}^{N_1} B_{\delta/2}(b^j_\infty))=\int_{B_r(x^{i_0}_\infty)\setminus \cup_{j=1}^{N_1} B_{\delta/2}(b^j_\infty)} e^{2\lambda_k} dvol_{h_k}\\[5mm]
\ds\quad\quad\quad\geq e^{-2C_{r,\delta}} \int_{B_r(x^{i_0}_\infty)\setminus \cup_{j=1}^{N_1} B_{\delta/2}(b^j_\infty)} e^{2\bar\lambda ^{i_0} _k} dvol_{h_k}
\end{array}
$$
contradicting the area bound, so (a) is proved. 

For getting $b)$ observe that if for one index $i_0$ there exists a constant such that $\bar \lambda^{i_0}_k\geq C>-\infty$ for infinitely many $k$, then by \eqref{eq:3.22} all the balls of the covering with non empty intersection with such a ball have the same property (the holes $\cup_{j=1}^{N_1} B_{\delta/2}(b^j_\infty)$ do not cover the intersection of two balls of the covering if $\delta>0$ is taken small enough) then, since $\Sigma$ is connected, on all the balls of the covering we have  $\bar \lambda^i_k\geq C$ and $\lambda_k$ are uniformly bounded on the compact subset $K$; on the other hand, if on every ball of the covering $\bar \lambda^i_k\to -\infty$ (up to subsequences in $k$) we have by \eqref{eq:3.22} that $\lambda_k \to -\infty$ uniformly on the compact subset $K$. Therefore also (b) is proved.
\end{proof}

\section{A ``good gauge extraction'' in the M\"obius group for sequences of immersions of spheres under $G-$energy control and
diameter positive lower bound. }
\reset
In the present subsection we prove that, assuming the reference surface is $\Sp^2$ and that the images $\Phi_k(\Sp^2)$ have diameter bounded below by a striclty positive costant, we can reparametrize the immersions such that in Proposition \ref{Prop:ConcComp} the degenerating case $\lambda_k\to-\infty$ does not occur; therefore on compact subsets invading $\Sp^2$ there is an $L^\infty$ control of the conformal factors. 

\begin{Lm}
\label{lm-IV.1}{\bf [Good gauge extraction Lemma].}
Let  ${\cal F}_{\Sp^2}\ni\bP_k\hookrightarrow (M^m,h)$ be a sequence of weak branched conformal immersions of $\Sp^2$ into the $m$-dimensional riemannian manifold $(M^m,h)$ and assume the following conditions hold true:
\begin{itemize}
\item[{\rm (i)}] the areas of $\Sp^2$ induced by the immersions are uniformly bounded from above: called $g_k:=\bP_k ^*(h)$
$$\Area_{g_k}(\Sp^2)\leq C;$$
\item[{\rm (ii)}] the energies of the $\bP_k$ are uniformly bounded from above:
$$F(\bP_k)=\frac{1}{2} \int_{\Sp^2} |\D \bn_{\bP_k}|^2 dvol_{g_k}= \frac{1}{2} \int_{\Sp^2} | {\mathbb I}_{\bP_k}|^2 dvol_{g_k} \leq C;$$ 
\item[{\rm (iii)}] the diameters of $\bP_k(\Sp^2)$ are bounded below by a strictly positive constant:
$$\diam_{M}(\bP_k(\Sp^2))\geq \frac{1}{C}.$$
\end{itemize}
Then for every $k$ (up to subsequences) there exists a positive M\"obius transformations $f_k\in{\mathcal M}^+(\Sp^2)$   such that called ${\cal F}_{\Sp^2}\ni\tilde{\bP}_k:=\bP_k \circ f_k$ the reparametrized immersion and $\tilde{\lambda}_k:=\log|\partial_{x_1} \tilde{\bP}_k|=\log|\partial_{x_2} \tilde{\bP}_k|$ the new conformal factor the following holds:

There exists a finite set of points $\{a_1, \ldots, a_N\}\subset \Sp^2$ such that for any compact subset $K\subset \subset  \Sp^2 \setminus\{a_1,\ldots,a_N\}$ there exists a constant $C_K$ depending on the compact $K$ and on the bounds on areas and energies of $\bP_k$ such that, up to subsequences on $k$,   
 $$\|\tilde{\lambda}_k\|_{L^\infty(K)}\leq C_K\quad .$$\hfill $\Box$
\end{Lm}

\begin{proof}
We prove the lemma by contradiction. If the thesis is not true then for every $f_k:\Sp^2\to \Sp^2$ composition of isometries and dilations, called $\tilde{\bP}_k:=\bP_k \circ f_k$ the reparametrized immersion and $\tilde{\lambda}_k:=\log|\partial_{x_1} \tilde{\bP}_k|=\log|\partial_{x_2} \tilde{\bP}_k|$ the new conformal factor, the second case in (b) of Proposition \ref{Prop:ConcComp} occurs: there exists a finite set of points $\{a_1, \ldots, a_N\}\subset \Sp^2$ such that for any compact subset $K\subset \subset  \Sp^2 \setminus\{a_1,\ldots,a_N\}$, up to subsequences on $k$,  
 $$\tilde{\lambda}_k \to -\infty \text { uniformly on } K.$$
 On the other hand, called $\vec{I}:M\to \Rp$ the Nash isometric embedding, combining assumptions (i) and (ii) as in the begininning of the proof of Proposition \ref{Prop:ConcComp} we have \eqref{eq:bdd2ffRp1}, namely
 \be\label{eq:bdd2ffRp2}
\int_{\Sp^2} |d \bn_{\vec{I}\circ \bP_k}|_{g_k}^2 dvol_{g_k}\leq C.
\ee
We will prove that if the thesis is not true and assumption (iii) holds, then \eqref{eq:bdd2ffRp2} cannot be satisfied.

Fix any $ 2 \leq n \in \N$, take $p^1_k,\ldots,p^n_k \in \bP_k(\Sp^2)$ such that 
\be\label{eq:LowerDist}
\dist_M(p^i_k,p^j_k)\geq \frac{1}{Cn} \quad \forall i\neq j
\ee
(clearly this is can be done thanks to assumption (iii) ) and pick $q^i_k\in \Sp^2$ such that $\bP_k(q^i_k)=p^i_k$ (observe that the order in the $i$ indicization is not relevant; in the sequel we will relabel the points for convenience of notation). Considering stereographic projection with center different from all the points $\{q^i_k\}_{i,k}$ we can carry on the proof in $\R^2$.

STEP 1: for every $k$ consider the two points at minimal distance; up to subsequences in $k$ and relabeling in $i$ we can assume they are $q^1_k$ and $q^2_k$:
$$|q^1_k-q^2_k|_{\R^2}\leq |q^i_k-q^j_k|_{\R^2} \quad \forall k, \quad \forall i\neq j.$$
 Up to compositions $f_k^{-1}:\R^2 \to \R^2$ of isometries and dilations, we can assume that $q^1_k=(0,0)$ is the origin and $q^2_k=(1,0)$. By construction, for fixed $k$, all the  points $q^i_k$ are at mutual distance at least $1$ and up to subsequences, either $|q^1_k|\to \infty$ as $k\to \infty$ or  
$$q^i_k \to q^i_{\infty} \quad \text{with} \quad |q^i_{\infty}-q^j_{\infty}|\geq 1 \quad \forall i\neq j.$$
Since we are assuming the thesis is not true, as explained above, we have that there exists a finite set of points $\{a_1, \ldots, a_N\}\subset \Sp^2$ such that for any compact subset $K\subset \subset  \Sp^2 \setminus\{a_1,\ldots,a_N\}$, up to subsequences on $k$,  
 \be \label{eq:lambdainfty2}
 \lambda_k \to -\infty \text { uniformly on } K;
 \ee
(with abuse of notation we identify $\bP_k$ with $\tilde{\bP}_k$ and $\lambda_k$ with $\tilde{\lambda}_k$).
\\It is clear that one can construct two smooth open subsets $C^1\subset B_{\frac{2}{3}}((0,0))$ and $C^2\subset B_{\frac{2}{3}}((1,0))$ such that
\begin{itemize}
\item $\bar{C}^1$ and $\bar{C}^2$ are diffeomorphic to closed balls,
\item the intersection of $\bar{C}^1$ and $\bar{C}^2$ consists of just one point, say $q_0$:
\be\label{eq:intersection1}
\bar{C}^1 \cap \bar{C}^2=\{q_0\}=\partial C^1\cap \partial C^2,
\ee
\item $q^1_k=(0,0)\in C^1$ and $q^2_k=(1,0)\in C^2$ but observe that by construction  $\dist_{\R^2} (q^i_k, C^1\cup C^2)\geq \frac{1}{3}$ for every $i\geq 3$,
\item $\left(\partial C^1 \cup \partial C^2\right) \cap \{a_1,\ldots,a_N\}=\emptyset$.
\end{itemize}
Observe that the last condition together with \eqref{eq:lambdainfty2} implies that the lengths of the images of $\partial C^1$ and $\partial C^2$ converge to 0:
\be\label{eq:partialCto0}
{\cal H}^1\left(\bP_k (\partial C^1) \cup \bP_k (\partial C^2) \right) \to 0. 
\ee
Recall that the branch points $\{b^1_k,\ldots,b^{N_k}_k\}$ of $\Phi_k$ converge up to subsequences to a subset of $\{a_1,\ldots,a_N\}$, then we also have  (this information will be used later in Step 2)
\be\label{eq:NoBranchC}
\left(\partial C^{1}\cup \partial C^2\right)\cap \{b^1_k,\ldots,b^{N_k}_k\}=\emptyset \quad \text{for large } k.
\ee
Notice moreover that at least one between $\bP_k(C^1)$ and $\bP_k(C^2)$ has diameter at least $\frac{1}{2nC}$, say $\bP_k(C^1)$:
\be\label{eq:lbd}
\diam_M(\bP_k(C^1))\geq \frac{1}{2nC}.
\ee
 Indeed if $\diam_M(\bP_k(C^1))<\frac{1}{2nC}$ and $\diam_M(\bP_k(C^2))<\frac{1}{2nC}$, since $\bP_k(C^1)$ and $\bP_k(C^2)$ are connected throughout $\bP_k(q_0)$ then $\dist_M(p^1_k,p^2_k)=\dist_M(\bP_k(q^1_k),\bP_k(q^2_k))<\frac{1}{nC}$ contradicting \eqref{eq:LowerDist}.

Now recall Lemma I.1 in \cite{Ri4} for weak immersions without branch points: let $\Sigma$ be a compact surface with boundary and let $\bP:\Sigma \hookrightarrow \Rp$ be a weak immersion in $\Rp$  without branch points; then the following inequality holds 
\be\label{eq:MFBoundary1}
4 \pi\leq \int_{\Sigma} |\vec{H}|^2 d\mu_g + 2 \frac{{\cal H}^1(\bP (\partial \Sigma))}{d( \bP(\partial\Sigma),\bP(\Sigma))}
\ee
where $d$ is the usual distance between two sets:  $d( \bP(\partial\Sigma),\bP(\Sigma)):=\sup_{p_1\in \bP(\Sigma)} \inf_{p_2 \in  \bP(\partial \Sigma)} |p_1-p_2|_{\R^p}$.

Let $b^1_k,\ldots,b^{L_k}_k$ be the branch points of $\bP_k$ contained in $C^1$, then apply \eqref{eq:MFBoundary1} to $\bP_k$ restricted to $\bar{C}^1\setminus (\cup_{j=1}^{L_k} B_\epsilon(b^j_k))$ (by construction $b^j_k$ converges to some $a_j\notin\partial C^1$ as $k\to \infty$ so for large $k$ and small $\epsilon$, $B_\epsilon(b^j_k)\subset C^1$ and the difference set above is smooth; clearly this restriction of $\bP_k$ is a weak immersion without branch points); as $\epsilon \to 0$ we obtain 
\be\label{eq:MFBoundary}
4 \pi\leq \int_{C^1} |\vec{H}_{\vec{I}\circ \bP_k}|^2 dvol_{g_k} + 2 \frac{{\cal H}^1(\bP_k (\partial C^1))}{d( \bP_k(\partial C^1)\bigcup (\cup_{j=1}^{L_k} \bP_k(b^j_k)),\bP_k(C^1))}.
\ee

Observe that 
\be\label{eq:d>0}
\liminf_{k\to \infty} d\left( \bP_k(\partial C^1)\bigcup (\cup_{j=1}^{L_k} \bP_k(b^j_k)),\bP_k(C^1)\right) >0.
\ee
Indeed, if up to subsequences $d( \bP_k(\partial C^1)\bigcup (\cup_{j=1}^{L_k} \bP_k(b^j_k)),\bP_k(C^1))\to 0$ then, since by \eqref{eq:partialCto0} we have ${\cal H}^1 (\bP_k(\partial C^1))\to 0$, it is easy to see that  we would have also $\diam_M(\bP_k(C^1))\to 0$; contradicting \eqref{eq:lbd}.

Now, using \eqref{eq:d>0} and \eqref{eq:partialCto0}, formula \eqref{eq:MFBoundary} gives
\be\label{eq:1}
4\pi - \frac{1}{n-1} \leq \int_{C^1} |\vec{H}_{\vec{I}\circ \bP_k}|^2 dvol_{g_k} \quad \text{for large } k. 
\ee
In other words, for large $k$ we isolated a region $C^1\subset \R^2$ containing $q^1_k$, with distance at least $\frac{1}{3}$ from all the other points $q^i_k$, with Willmore energy at least  $4\pi - \frac{1}{n-1}$.
 \newline
 \newline
STEP 2: we continue the proof by induction.

\emph{Base of the induction}. Let $2\leq i_0\leq n-1$ and assume the following properties are satisfied:
\\1) for every $ i<i_0$ and every $k$ up to subsequences there exists a smooth open subset $C^i_k\subset B_{r_i}(q^i_k)\subset \R^2$, $r_i>0$ independent on $k$, such that  
\begin{itemize}
\item [{\rm 1a)}] $q^i_k \in C^i_k$ but $q^j_k \notin \bar{C}^i_k$ for every $i,j< i_0, i\neq j$,
\item [{\rm 1b)}] ${\cal H}^1(\bP_k(\partial C^i_k))\to 0$ as $k\to \infty$,
\item [{\rm 1c)}]  $\partial C^{i}_k \cap \{b^1_k,\ldots, b^{N_k}_k\}=\emptyset$ where, as before, $b^1_k,\ldots ,b^{N_k}_k$ are the branch points of $\bP_k$,
\item [{\rm 1d)}] $4\pi - \frac{1}{n-1} \leq \int_{C^i_k} |\vec{H}_{\vec{I}\circ \bP_k}|^2 dvol_{g_k}$,
\item [{\rm 1e)}]  $C^i_k\cap C^j_k=\emptyset$ for every $i,j < i_0, i\neq j$, 
\item [{\rm 1f)}] if there exists $i_1< i_0$ such that for some $j\geq i_0$ it happens that $\dist_{\R^2}(q^j_k,C^{i_1}_k)\to 0$ as $k\to \infty$ up to subsequences, then for every $i\leq i_1$ we have $\diam_{\R^2} (C^i_k)\to 0$ as $k\to \infty$ on the same subsequence,  
\end{itemize}
2) $|q^i_k-q^j_k|_{\R^2}\geq 1$ for every $i,j \geq i_0$. 
\newline
\newline
\emph{Inductive step}. For every $k$ there exists a smooth open subset $C^{i_0}_k\subset \R^2$  and there exists $f_k:\R^2 \to \R^2$ composition of dilations and isometries such that called $\tilde{q}^i_k:=f_k^{-1}(q^i_k)$ and $\tilde{C}^i_k:=f_k^{-1}(C^i_k)$ the following properties hold:
\begin{itemize}
\item [{\rm 1a)}] up to relabeling in $j$  the points $\tilde{q}^j_k$, $j\geq i_0$ and up to subsequences in $k$ we have $\tilde{q}^{i_0}_k \in C^{i_0}_k$ but $\tilde{q}^j_k \notin \bar{C}^{i_0}_k$ for every $j\neq i_0$, moreover $C^{i_0}_k\subset B_{r_{i_0}}(\tilde{q}^{i_0}_k)\subset \R^2$ with $r_{i_0}$ independent of $k$, 
\item [{\rm 1b)}] ${\cal H}^1(\bP_k(\partial C^{i_0}_k))\to 0$ as $k\to \infty$,
\item [{\rm 1c)}]  $\partial C^{i_0}_k \cap \{\tilde{b}^1_k,\ldots, \tilde{b}^{N_k}_k\}=\emptyset$ where  $\tilde{b}^1_k:=f_k^{-1}({b}^1_k),\ldots ,\tilde{b}^{N_k}_k:=f_k^{-1}({b}^{N_k}_k)$ are the branch points of $\tilde{\bP}_k:=\bP_k\circ f_k$,
\item [{\rm 1d)}] $4\pi - \frac{1}{n-1} \leq \int_{C^{i_0}_k} |\vec{H}_{\vec{I}\circ \bP_k}|^2 dvol_{g_k}$,
\item [{\rm 1e)}] $C^{i_0}_k\cap \tilde{C}^j_k=\emptyset$ for every $j < i_0$,
\item [{\rm 1f)}] $\liminf_k \dist_{\R^2} (C^{i_0}_k,\tilde{q}^j_k)>0$ for every $j>i_0$,
\end{itemize}
2) $|\tilde{q}^i_k-\tilde{q}^j_k|_{\R^2}\geq 1$ for every $i,j \geq i_0+1$. 
\newline
\newline
Let us prove the inductive step assuming the base of the induction is satisfyed.

As in Step 1 for every $k$ consider the two points at minimal distance among $\{q^i_k\}_{i\geq i_0}$; up to subsequences in $k$ and relabeling in $i$ we can assume they are $q^{i_0}_k$ and $q^{i_0+1}_k$:
$$|q^{i_0}_k-q^{i_0+1}_k|_{\R^2}\leq |q^i_k-q^j_k|_{\R^2} \quad \forall k, \quad \forall i,j\geq i_0, i\neq j .$$
 Consider $f_k:\R^2 \to \R^2$ composition of isometries and dilations such that $\tilde{q}^{i_0}_k:=f_k^{-1}(q^{i_0}_k)=(0,0)$ is the origin and $\tilde{q}^{i_0+1}_k:=f_k^{-1}(q^{i_0+1}_k)=(1,0)$. By construction, for fixed $k$, all the  points $\{\tilde{q}^i_k\}_{i\geq i_0}$ are at mutual distance at least $1$ and up to subsequences either $|\tilde{q}^i_k| \to \infty$ as $k\to \infty$ or
 \be\label{eq:lbdistance}
\tilde{q}^i_k \to \tilde{q}^i_{\infty} \quad \text{with} \quad |\tilde{q}^i_{\infty}-\tilde{q}^j_{\infty}|\geq 1 \quad \forall i,j\geq i_0, i\neq j.
\ee
Since we are assuming the thesis is not true, as in Step 1, we have that there exists a finite set of points $\{a_1, \ldots, a_N\}\subset \Sp^2$ such that for any compact subset $K\subset \subset  \Sp^2 \setminus\{a_1,\ldots,a_N\}$, up to subsequences on $k$,  
 \be \label{eq:lambdainfty}
 \tilde{\lambda}_k \to -\infty \text { uniformly on } K
 \ee
 where $\tilde{\lambda}_k=\log|\partial_{x_1} \tilde{\bP}_k|=\log|\partial_{x_2} \tilde{\bP}_k|$ is the conformal factor of $\tilde{\bP}_k:=\bP_k\circ f_k$.
Observe that if $\limsup_{k \to \infty} |q^{i_0}_k-q^{i_0+1}_k|<+\infty$ then the situation described in 1f) of the base of induction is exactly the same for the rescaled quantities $\tilde{q}^j_k, \tilde{C}^i_k$. On the contrary, if for some subsequence in $k$ we have $|q^{i_0}_k-q^{i_0+1}_k|\to +\infty$ then we are rescaling with a diverging ratio and, on the same subsequence, $\lim_{k\to \infty} \diam_{\R^2} (\tilde{C}^i_k)=0$ for all $i<i_0$. Therefore, in both cases, if for some $i_1<i_0$ and some subsequence in $k$ it happens that $\lim_{k\to \infty} \dist_{\R^2} (\tilde{C}^{i_1}_k,\tilde{q}^{i_0}_k)=0$ or $\lim_{k\to \infty} \dist_{\R^2} (\tilde{C}^{i_1}_k,\tilde{q}^{i_0+1}_k)=0$ then on the same subsequence 
\be \label{eq:DiamCTo0}
\lim_{k\to \infty} \diam_{\R^2} (\tilde{C}^i_k)=0 \quad \forall i \leq i_1.
\ee
Hence we are left to two possibilities:

Case a): $\tilde{q}^i_{\infty}\neq (0,0),(1,0)$ ( recall that $(0,0)=\tilde{q}^{i_0}_k, (1,0)=\tilde{q}^{i_0+1}_k$ for every $k$). Then by the previous discussion, on the considered subsequence in $k$, we have 

\be \label{eq:casea}
\liminf_{k\to \infty} \dist_{\R^2} (\tilde{C}^{i}_k,\tilde{q}^{i_0}_k)>0 \quad \text{ and } \quad \liminf_{k\to \infty} \dist_{\R^2} (\tilde{C}^{i}_k,\tilde{q}^{i_0+1}_k)>0 \quad \text{ for all } i<i_0.
\ee

Case b): for some $i_1<i_0$ we have $\tilde{q}^{i_1}_{\infty}= (0,0)$ or $\tilde{q}^{i_1}_{\infty}=(1,0)$ then, by the previous discussion, \eqref{eq:DiamCTo0} holds.

Let us first consider case a). The situation is analogous to Step 1 since for all $i<i_0$ the sets  $\tilde{C}^i_k$ are at uniform strictly positive distance from $\tilde{q}^{i_0}_k$ and $\tilde{q}^{i_0+1}_k$. From the iterative construction of $\tilde{C}^i_k$ it is clear that, for every $k$, the set $(\R^2\setminus\bigcup_{i<i_0} \bar{\tilde{C}}^i_k)\ni \{\tilde{q}^{i_0}_k,\tilde{q}^{i_0+1}_k\}$ is connected and that it is possible to construct two smooth open subsets $C^{i_0}$ and $C^{i_0+1}$ such that
\begin{itemize}
\item $\bar{C}^{i_0}$ and $\bar{C}^{i_0+1}$ are diffeomorphic to closed balls,
\item the intersection of $\bar{C}^{i_0}$ and $\bar{C}^{i_0+1}$ consists of just one point, say $q_0$:
\be\label{eq:intersection}
\bar{C}^{i_0} \cap \bar{C}^{i_0+1}=\{q_0\}=\partial C^{i_0}\cap \partial C^{i_0+1},
\ee
\item $\tilde{q}^{i_0}_k=(0,0)\in C^{i_0}$ and $\tilde{q}^{i_0+1}_k=(1,0)\in C^{i_0+1}$,   
\item $\left(\partial C^{i_0} \cup \partial C^{i_0+1}\right) \cap \{a_1,\ldots,a_N\}=\emptyset$, and $\left(\partial C^{i_0}\cup \partial C^{i_0+1}\right)\cap \{\tilde{b}^1_k,\ldots,\tilde{b}^{N_k}_k\}=\emptyset$, where  $\{\tilde{b}^1_k,\ldots,\tilde{b}^{N_k}_k\}$ are the  branch points of $\tilde{\bP}_k=\bP_k\circ f_k$,
\item $\liminf_{k \to \infty } \dist_{\R^2} (\tilde{q}^j_k, C^{i_0}\cup C^{i_0+1})> 0$ for every $j> i_0+1$  (this can be done thanks to \eqref{eq:lbdistance}),
\item $\liminf_{k\to \infty} \dist_{\R^2}(\tilde{C}^i_k, C^{i_0} \cup C^{i_0+1})>0$ for all $i<i_0$ (this can be done thanks to \eqref{eq:casea}).
\end{itemize}

Now exactly as in Step 1, we have that ${\cal H}^1\left(\bP_k (\partial C^{i_0}) \cup \bP_k (\partial C^{i_0+1}) \right) \to 0$ and  at least one between $\bP_k(C^{i_0})$ and $\bP_k(C^{i_0+1})$ has diameter at least $\frac{1}{2nC}$, say $\bP_k(C^{i_0})$:
$\diam_M(\bP_k(C^{i_0}))\geq \frac{1}{2nC}$.
As in Step 1 we can pass to the limit in the inequality \eqref{eq:MFBoundary} and conclude that
\be\label{eq:}
4\pi - \frac{1}{n-1} \leq \int_{C^{i_0}} |\vec{H}_{\vec{I}\circ \bP_k}|^2 dvol_{g_k} \quad \text{for large } k. 
\ee
In other words, for large $k$ we isolated another region $C^{i_0}\subset \R^2$ containing $q^{i_0}_k$, disjoint from the previous $\{\tilde{C}^i_k\}_{i<i_0}$, with Willmore energy at least  $4\pi - \frac{1}{n-1}$. Observe that by construction 1a)$\ldots$1f) and 2) of the inductive step are satisfied.
 \newline

Proof in case b). Let us call $I_0:=\{i<i_0 \text{ such that }\tilde{q}^i_{\infty}=(0,0)=\tilde{q}^{i_0}_k\}$, $I_1:=\{i<i_0 \text{ such that }\tilde{q}^i_{\infty}=(1,0)=\tilde{q}^{i_0+1}_k\}$ and $I_2:=\{1,\ldots,i_0-1\}\setminus (I_0\cup I_1)$; by assumption at least one between $I_0$ and $I_1$ is non empty. Observe that, up to subsequences in $k$, we have 
\begin{eqnarray}
\forall i \in I_0 &&   \dist_{\R^2} (\tilde{C}^{i}_k,\tilde{q}^{i_0}_k)\to 0 \quad \text{and} \quad \diam_{\R^2} (\tilde{C}^i_k)\to 0, \label{I0}\\
\forall i \in I_1 &&  \dist_{\R^2} (\tilde{C}^{i}_k,\tilde{q}^{i_0+1}_k)\to 0 \quad \text{and} \quad \diam_{\R^2} (\tilde{C}^i_k)\to 0, \label{I1}\\
\forall i \in I_2 &&   \liminf_{k\to \infty} \dist_{\R^2} (\tilde{C}^{i}_k,\tilde{q}^{i_0}_k)> 0 \quad \text{and} \quad \liminf_k \dist_{\R^2} (\tilde{C}^{i}_k,\tilde{q}^{i_0+1}_k)> 0. \label{I2}
\end{eqnarray}
At first we do not consider $I_0$ and $I_1$  and construct $C^{i_0}$ and $C^{i_0+1}$ as in case a) satisfying the same itemization with the only difference that in the last item we ask  (it can be done thanks to \eqref{I2})
$$\liminf_{k\to \infty} \dist_{\R^2}(\tilde{C}^i_k, C^{i_0} \cup C^{i_0+1})>0 \quad  \forall i\in I_2.$$
Now, for every $k$, call
\be\label{def:Cik}
C^{i_0}_k:= C^{i_0}\setminus \left( \bigcup_{i\in I_0} \bar{\tilde{C}}^i_k \right) \quad \text{and} \quad C^{i_0+1}_k:=C^{i_0+1}\setminus \left( \bigcup_{i\in I_1} \bar{\tilde{C}}^i_k \right).
\ee
Observe that, by 1a),  $\tilde{q}^{i_0}_k=(0,0)\in C^{i_0}_k$ and $\tilde{q}^{i_0+1}_k=(1,0)\in C^{i_0+1}_k$ and, by the construction of $C^{i_0},C^{i_0+1}$ and since by $1b)$ we have ${\cal H}^1(\tilde{\bP}_k(\partial \tilde{C}^i_k))={\cal H}^1(\bP_k(\partial {C}^i_k)) \to 0$ for all $i<i_0$, we still have that
\be \label{eq:H1to0}
 {\cal H}^1(\tilde{\bP}_k(\partial {C}^{i_0}_k)) \to 0 \quad \text{and} \quad {\cal H}^1(\tilde{\bP}_k(\partial {C}^{i_0+1}_k)) \to 0.
\ee

Moreover, exactly as before, at least one between $\tilde{\bP}_k(C^{i_0}_k)$ and $\tilde{\bP}_k(C^{i_0+1}_k)$ has diameter at least $\frac{1}{2nC}$, say $\tilde{\bP}_k(C^{i_0}_k)$:
\be \label{eq:lbdCi0k}
\diam_M(\bP_k(C^{i_0}_k))\geq \frac{1}{2nC}.
\ee 
Now observe that by construction of $C^{i_0}$ and by assumption 1c), for every $k$ we have 
$$\partial C^{i_0}_k\cap \{\tilde{b}^1_k,\ldots,\tilde{b}^{N_k}_k\}=\emptyset$$
where $\{\tilde{b}^1_k,\ldots,\tilde{b}^{N_k}_k\}$ are the branch points of $\tilde{\bP}_k$. Then, manipulating formula \eqref{eq:MFBoundary1} as in Step 1 for the branched case, called $\{\tilde{b}^1_k,\ldots,\tilde{b}^{L_k}_k\}$  the branch points contained in $C^{i_0}_k$ we obtain that
\be\label{eq:MFBoundary2}
4 \pi\leq \int_{C^{i_0}_k} |\vec{H}_{\vec{I}\circ \tilde{\bP}_k}|^2 dvol_{g_k} + 2 \frac{{\cal H}^1(\tilde{\bP}_k (\partial C^{i_0}_k))}{d( \bP_k(\partial C^{i_0}_k)\bigcup (\cup_{j=1}^{L_k} \tilde{\bP}_k(\tilde{b}^j_k)),\tilde{\bP}_k(C^{i_0}_k))}.
\ee
As before, since $\diam_M(\tilde{\bP}_k(C^{i_0}_k))\geq\frac{1}{2nC}$ and ${\cal H}^1(\tilde{\bP}_k (\partial C^{i_0}_k)) \to 0$, then 
$$\liminf_{k \to \infty} d\left( \bP_k(\partial C^{i_0}_k)\bigcup (\cup_{j=1}^{L_k} \tilde{\bP}_k(\tilde{b}^j_k)),\tilde{\bP}_k(C^{i_0}_k)\right)>0$$ and passing to the limit for $k\to \infty$ in \eqref{eq:MFBoundary2} we get
\be\label{eq:finalCio}
4\pi - \frac{1}{n-1} \leq \int_{C^{i_0}_k} |\vec{H}_{\vec{I}\circ \tilde{\bP}_k}|^2 dvol_{g_k} \quad \text{for large } k. 
\ee
In other words, for large $k$ we isolated another region $C^{i_0}_k\subset \R^2$ containing $q^{i_0}_k$, disjoint from the previous $\{\tilde{C}^i_k\}_{i<i_0}$, with Willmore energy at least  $4\pi - \frac{1}{n-1}$. Observe that by construction 1a)$\ldots$1f) and 2) are satisfied, so we completed the proof of the inductive step.
 \newline
 \newline
 Let us summarize and conclude the proof: we showed that if the base of the induction is satisfyed for a certain $2\leq i_0\leq n-1$ then we proved that the inductive step is true and so, from the procedure described above, it is clear that the base of induction is satisfied also for $i_0+1$. The iteration procedure stops for $i_0=n-1$ and at that point, for large $k$, we constructed $n-1$ disjoint subsets $C^{1}_k,\ldots,C^{n-1}_k$ each one carrying a Willmore energy of at least $4 \pi-\frac{1}{n-1}$ then
 \be\label{eq:contrad}
 \int_{\Sp^2}  |d \bn_{\vec{I}\circ \bP_k}|_{g_k}^2 dvol_{g_k}\geq \int_{\bigcup_{i=1}^{n-1}C^{i}_k} |d \bn_{\vec{I}\circ \bP_k}|_{g_k}^2 dvol_{g_k}\geq 4\pi(n-1)-1.
 \ee
 Since $n$ is arbitrary large, the lower bound \eqref{eq:contrad} clearly contradicts the upper bound \eqref{eq:bdd2ffRp2}.  
\end{proof}

\section{A Diameter Tracking Procedure.}
\reset
The purpose of the present section is to prove the following lemma.

\begin{Lm}
\label{lm-V.1}{ \bf[Diameter tracking procedure]}
Let $\{\vec{\Phi}_k\}_{k \in \N}\subset {\mathcal F}_{\Sp^2}$ be a sequence of conformal weak, possibly branched, immersions into $M^m$. Assume that
\be
\label{V.a.l.1}
\limsup_{k\rightarrow +\infty}\int_{\Sp^2}\lf[1+|D\vec{n}_{\vec{\Phi}_k}|_h^2\rg]\ dvol_{g_k}<+\infty\quad,
\ee
where $dvol_{g_k}$ denotes the volume form associated to the induced metric $g_k:=\vec{\Phi}_k^\ast h$ by $\vec{\Phi}_k$ on 
$\Sp^2$ and $\vec{H}_{\vec{\Phi}_k}$ is the mean curvature vector associated to the immersion $\vec{\Phi}_k$. Let $a\in \Sp^2$, $\delta_k\rightarrow 0$ and $\ep_k\rightarrow 0$
and a finite family of sequences of points $(a^i_k)_{i=1\cdots N}$ together with a finite family of sequences of positive radii $(r_k^i)_{i=1\cdots N}$
satisfying the following conditions
\begin{eqnarray}
&&\forall\, i\in\{1\cdots N\}\quad\quad\lim_{k\rightarrow +\infty}\frac{|a_k^i-a|}{\ep_k\,\delta_k}+\frac{r^j_k}{\ep_k^2\,\delta_k}=0\quad \label{V.a.b.1}, \\
&&\forall\, i\ne j\quad\quad\lim_{k\rightarrow +\infty}\frac{|a_k^i-a_k^j|}{\ep_k^{-1}(r_k^i-r_k^j)} =0 \label{V.a.c.1} \\
&&\lim_{k\rightarrow +\infty}\int_{B_{\delta_k}(a)\setminus B_{\ep_k\, \delta_k}(a)}\lf[1+|{\mathbb I}_{\vec{\Phi}_k}|_h^2\rg]\ dvol_{g_k}=0 \label{V.a.d.1}\\
&&\forall\, i\in\{1\cdots N\}\quad\quad\lim_{k\rightarrow +\infty}\int_{B_{\ep_k^{-1}r^i_k}(a_k^i)\setminus B_{r_k^i}(a_k^i)}\lf[1+|{\mathbb I}_{\vec{\Phi}_k}|_h^2\rg]\ dvol_{g_k}=0\quad.\label{V.a.e.1}
\end{eqnarray}
Assume that $\vec{\Phi}_k$ has no branched points in each annulus $B_{\al^{-1}\,r^i_k}(a_k^i)\setminus B_{r_k^i}(a_k^i)$ 
as well as in $B_{\delta_k}(a)\setminus B_{\al\delta_k}(a)$ for any $0<\al<1$ and for $k$ large enough . Suppose
\be
\label{V.a.h.1}
\liminf_{k\rightarrow +\infty}diam\lf(\vec{\Phi}_k\lf(B_{\ep_k\, \delta_k}(a)\setminus \bigcup_{i=1}^N B_{\ep_k^{-1}r^i_k}(a_k^i)\rg)\rg)>0\quad,
\ee
then, modulo extraction of a subsequence, there exists a sequence of positive M\"obius transformations $f_k\in{\mathcal M}^+(\Sp^2)$, there exists $Q\in {\N}$ and  $Q$ points
$b_1,\cdots, b_Q$ and $\vec{\xi}_\infty\in {\mathcal F}_{\Sp^2}$ such that
\be
\label{V.a.f.1}
\vec{\xi}_k:=\vec{\Phi}_k\circ f_k\rightharpoonup \vec{\xi}_\infty\quad\quad\mbox{ weakly in }W^{2,2}_{loc}(\Sp^2\setminus\{b_1,\cdots, b_Q\}),
\ee
moreover for  any compact $K\subset \Sp^2\setminus\{b_1\cdots b_Q\}$
\be
\label{V.a.g.1}
\limsup_{k\rightarrow +\infty}\|\log| d(\vec{\Phi}_k\circ f_k)|\|_{L^\infty(K)}<+\infty\quad.
\ee
Finally there exists $s_k\rightarrow 0$ such that
\be
\label{V.a.j.1}
\vec{\Phi}_k\circ f_k\longrightarrow \vec{\xi}_\infty\quad\quad\mbox{ unif. in $C^0$ on } \Sp^2\setminus \bigcup_{j=1}^Q B_{s_k}(b_j)\quad, 
\ee
and for any $i\in\{1\cdots N\}$ there exists $j_i\in\{1\cdots Q\}$
\be
\label{V.a.k.1}
f_k^{-1}(B_{r^i_k}(a_k^i))\subset B_{s_k}(b_{j_i})
\ee
moreover there exists also $j_0\in\{1\cdots Q\}$ such that
\be
\label{V.a.k.2}
f_k^{-1}(\Sp^2\setminus B_{\delta_k}(a))\subset B_{s_k}(b_{j_0})\quad.
\ee 
\hfill $\Box$
\end{Lm}
\noindent  Lemma~\ref{lm-V.1} is  a consequence of the combination of a so called ``cutting and filling Lemma'' together with the ``good-gauge extraction Lemma''~\ref{lm-IV.1}.
We first present the  ``cutting and filling Lemma'' and its proof before to end this section with the proof of lemma~\ref{lm-V.1}.
\begin{Lm} {\bf[Cutting and filling Lemma]}
\label{lm-V.2}
Let $\{\vec{\Phi}_k\}_{k \in \N} \subset {\mathcal F}_\Sp^2$ be a sequence of conformal weak, possibly branched, immersions $\vec{\Phi}_k$ of  into $M^m$. Assume that
\be
\label{V.a.1}
\limsup_{k\rightarrow +\infty}\int_{\Sp^2}\lf[1+|D\vec{n}_{\vec{\Phi}_k}|_h^2\rg]\ dvol_{g_k}<+\infty\quad,
\ee
where $dvol_{g_k}$ denotes the volume form associated to the induced metric $g_k:=\vec{\Phi}_k^\ast h$ by $\vec{\Phi}_k$ on 
$\Sp^2$ and $\vec{H}_{\vec{\Phi}_k}$ is the mean curvature vector associated to the immersion $\vec{\Phi}_k$. Let $a\in \Sp^2$ and $s_k, t_k\rightarrow 0$ be  such that
\be
\label{V.a.2}
\frac{t_k}{s_k}\longrightarrow 0\quad,
\ee
and
\be
\label{V.a.3}
\lim_{k\rightarrow +\infty}\int_{B_{s_k}(a)\setminus B_{t_k}(a)}\lf[1+|{\mathbb I}_{\vec{\Phi}_k}|_h^2\rg]\ dvol_{g_k}=0
\ee
and assume that $\vec{\Phi}_k$ has no branch points in $B_{s_k}(a)\setminus B_{\al\, s_k}(a)$ for any $0<\al<1$ and $k$ large enough. Then there exists a conformal immersion $\vec{\xi}_k$ from $\Sp^2$ into $M^m$ and a sequence of quasi conformal bilipshitz homeomorphisms $\Psi_k$ of $\Sp^2$,  converging in $C^{0}$ norm over $\Sp^2$
to the identity map, such that
\[
\vec{\xi}_k\circ\Psi_k=\vec{\Phi}_k\quad\quad\mbox{ in }\Sp^2\setminus B_{s_k}(a)
\]
and
\[
\lim_{k\rightarrow +\infty}diam (\vec{\xi}_k\circ\Psi_k(B_{s_k}(a))=0\quad, \quad \lim_{k\rightarrow +\infty}Area(\vec{\xi}_k\circ\Psi_k(B_{s_k}(a))=0
\]
moreover
\be
\label{V.a.3.b}
\lim_{k\rightarrow +\infty}\int_{B_{s_k}(a)}\lf[1+|{\mathbb I}^0_{{\vec{\xi}_k\circ\Psi_k}}|_h^2\rg]\ dvol_{g_{{\vec{\xi}_k\circ\Psi_k}}}=0\quad.
\ee
where ${\mathbb I}^0_{{\vec{\xi}_k\circ\Psi_k}}$ is the trace-free second fundamental form.
\hfill $\Box$
\end{Lm}
\noindent{\bf Proof of Lemma~\ref{lm-V.2}.} 
For  simplicity of the presentation we prove the lemma for a sequence of conformal immersions $\vec{\Phi}_k$ into the unit ball of ${\R}^m$ since the proof
for immersions into $M^m$ is fundamentally of identical nature. We apply successively the stereographic projection $\pi$ from $\Sp^2$ into ${\C}$ that sends
$a$ to the origin $0$ of the complex plane  and the dilation of radius $\simeq 1/s_k$ such that the composition of these two maps
is sending $B_{s_k}(a)\setminus B_{t_k}(a)$ into $B_1(0)\setminus B_{\rho_k}(0)$ with $\rho_k\rightarrow 0$. We keep denoting $\vec{\Phi}_k$ the conformal map
that we obtained after composing the original $\vec{\Phi}_k$ with these  two maps. Since for any $1>\al>0$ and $k$ large enough $\vec{\Phi}_k$ realizes a lipschitz conformal immersion of
the annulus $B_1(0)\setminus B_{\al}(0)$ into the unit ball of ${\R}^m$, there exists $\la_k\in L^\infty(B_1(0)\setminus B_{\al}(0))$ such that
\[
e^{\la_k}=|\p_{x_1}\vec{\Phi}_k|=|\p_{x_2}\vec{\Phi}_k|\quad\quad\mbox{ in }B_1(0)\setminus B_{\al}(0)\quad.
\]
From (\ref{V.a.3}) we deduce that
\be
\label{V.a.4}
\lim_{k\rightarrow+\infty}\int_{B_1(0)\setminus B_{\rho_k}(0)}|d\vec{n}_{\vec{\Phi}_k}|^2_{g_k}\ dvol_{g_k}=\lim_{k\rightarrow+\infty}\int_{B_1(0)\setminus B_{\rho_k}(0)}|\nabla \vec{n}_{\vec{\Phi}_k}|^2\ dx=0 \quad .
\ee
Using the argument in \cite{BR3} - first by Lemma  VI.1 of \cite{BR3} we extend $\vec{n}_{\Phi_k}$ in $B_{\rho_k}$ with energy control and then we use H\'elein's construction
of energy controlled moving frame, see Theorem V.2.1 in \cite{He} -  we deduce the existence of an orthonormal frame $(\vec{e}_{1,k},\vec{e}_{2,k})$ on $B_1(0)\setminus B_{\rho_k}(0)$  such that 
\[
\star(\vec{e}_{1,k}\wedge\vec{e}_{2,k})=\vec{n}_{\vec{\Phi}_k}
\]
and
\be
\label{V.a.5}
\int_{B_1(0)\setminus B_{\rho_k}(0)}\sum_{i=1}^2|\nabla\vec{e}_{i,k}|^2\ dx\le C\ \int_{B_1(0)\setminus B_{\rho_k}(0)}|\nabla \vec{n}_{\vec{\Phi}_k}|^2\ dx\quad.
\ee
A classical computation (see \cite{He} and \cite{Ri1}) gives  
\be
\label{V.a.5.a}
\Delta\la_k=\nabla^\perp\vec{e}_{1,k}\cdot\nabla\vec{e}_{2,k}\quad.
\ee
Let $\mu_k$ satisfy
\be
\label{V.a.6}
\lf\{
\begin{array}{l}
\ds\Delta\mu_k=\nabla^\perp\vec{e}_{1,k}\cdot\nabla\vec{e}_{2,k}\quad\quad\mbox{ in }B_1(0)\setminus B_{\al}(0)\\[5mm]
\ds\mu_k=0\quad\quad\mbox{ on }\p (B_1(0)\setminus B_{\al}(0)) \quad .
\end{array}
\rg.
\ee
Wente's inequality (see \cite{He}, \cite{Ge}, \cite{Top}) gives the existence of a constant $C>0$ independent of $\al$ such that
\be
\label{V.a.7}
\begin{array}{l}
\ds\|\mu_k\|_{L^\infty(B_1(0)\setminus B_{\al}(0)}+\|\nabla\mu_k\|_{L^2(B_1(0)\setminus B_{\al}(0)}\le C\ \|\nabla e_{1,k}\|_{L^2}\ \|\nabla e_{2,k}\|_{L^2}\\[5mm]
\ds\quad\quad\quad\quad\le C'\ \int_{B_1(0)\setminus B_{\rho_k}(0)}|\nabla \vec{n}_{\vec{\Phi}_k}|^2\ dx \quad.
\end{array}
\ee
Since from Lemma~\ref{lem:EstCFBranch} $\nabla\la_k$ is uniformly bounded in $L^{2,\infty}$, using (\ref{V.a.7}), we have that the harmonic function
$\nu_k:=\la_k-\mu_k$ satisfy
\be
\label{V.a.8}
\limsup_{k\rightarrow +\infty}\|\nabla \nu_k\|_{L^{2,\infty}(B_1(0)\setminus B_{\al}(0))}<+\infty\quad.
\ee
Standard elliptic estimates on harmonic function (see for instance \cite{GT}) imply that for any $2\al<\delta<1$
\[
\limsup_{k\rightarrow +\infty}\|\nu_k(x)-\nu_k(y)\|_{L^\infty((B_\delta(0)\setminus B_{2\al}(0))^2)}<+\infty \quad.
\]
Combining this fact with (\ref{V.a.7}) we obtain that there exists a constant $\ov{\la}_k$ satisfying
\be
\label{V.a.9.a}
\lim_{k\rightarrow +\infty}\ov{\la}_k=-\infty\quad,
\ee
and such that for any choice $2\al<\delta<1$
\be
\label{V.a.9}
\limsup_{k\rightarrow+\infty}\|\la_k-\ov{\la}_k\|_{L^\infty(B_\delta(0)\setminus B_{2\al}(0))}\quad.
\ee
We introduce the new map given by
\[
\hat{\vec{\Phi}}_k:= e^{-\ov{\la}_k}\ [\vec{\Phi}_k-\vec{\Phi}_k(0,1/2)]\quad.
\]
Because of (\ref{V.a.9}), it is clear that $\hat{\vec{\Phi}}_k(B_\delta(0)\setminus B_{2\al}(0))\subset B_{R_{\delta,\al}}(0)$ for some $R_{\delta,\al}>0$ and that $\hat{\vec{\Phi}}_k$
converges weakly in $(W^{1,\infty})^\ast$ to some non trivial limiting conformal immersion $\hat{\vec{\Phi}}_\infty$ of $B_\delta(0)\setminus B_{2\al}(0)$ :
\be
\label{V.a.10}
\hat{\vec{\Phi}}_k\rightharpoonup\hat{\vec{\Phi}}_\infty\quad\quad\mbox{ weakly in }(W^{1,\infty})^\ast(B_\delta(0)\setminus B_{2\al}(0))\quad.
\ee
Denote 
\[
\hat{\la}_k:=\log|\p_{x_1}\hat{\vec{\Phi}}_k|=\log|\p_{x_2}\hat{\vec{\Phi}}_k|\quad.
\]
Since $\hat{\la}_k$ is uniformly bounded in $L^\infty(B_\delta(0)\setminus B_{2\al}(0))$, since $\hat{\vec{\Phi}}_k$ satisfies
\[
\Delta\hat{\vec{\Phi}}_k=\frac{e^{2\,\hat{\la}_k}}{2}\,\vec{H}_{\hat{\vec{\Phi}}_k}
\]
and since the $L^2$ norm of $\vec{H}_{\hat{\vec{\Phi}}_k}$ is uniformly bounded (due to (\ref{V.a.4})) we deduce that
\be
\label{V.a.11}
\hat{\vec{\Phi}}_k\rightharpoonup\hat{\vec{\Phi}}_\infty\quad\quad\mbox{ weakly in }W^{2,2}(B_\delta(0)\setminus B_{2\al}(0))\quad.
\ee
From (\ref{V.a.4}) we have that for any  $2\al<\delta<1$
\[
\int_{B_\delta(0)\setminus B_{2\al}(0)}|\nabla\vec{n}_{\hat{\vec{\Phi}}_\infty}|^2\ dx=0 \quad .
\]
This implies that $\vec{n}_{\hat{\vec{\Phi}}_\infty}$ is constant on $D^2\setminus\{0\}$ equal to a unit simple $m-2$ vector $\vec{n}_0$, i.e. $\vec{n}_0=\vec{v}_1\ \wedge \ldots \wedge \vec{v}_{m-2}$ for some constant vectors $\vec{v}_1, \ldots, \vec{v}_{m-2}$ of $\R^m$. Thus $\hat{\vec{\Phi}}_\infty$ is conformal from  $D^2\setminus\{0\}$ into
a two dimensional plane $P^2_0$ that we identify to $z_i=0$ for $i\ge 3$ and identifies to an holomorphic map $f_\infty$. Denote $\hat{\la}_\infty$ the limit of $\hat{\la}_k$ 
\[
\hat{\la}_\infty=\log|\p_{x_1}\hat{\vec{\Phi}}_\infty|=\log|\p_{x_2}\hat{\vec{\Phi}}_\infty|=\log|f'(x_1+ix_2)| \quad ,
\]
which is harmonic in $D^2\setminus\{0\}$ - this can be deduced from (\ref{V.a.4}), (\ref{V.a.5}) and (\ref{V.a.5.a}) . Since again from Lemma~\ref{lem:EstCFBranch} $\nabla\la_k$ is uniformly bounded in $L^{2,\infty}$, we have 
\be
\label{V.a.12}
\|\nabla\hat{\la}_\infty\|_{L^{2,\infty}}<+\infty\quad.
\ee
Thus there exists $c_0\in {\R}$ such that
\be
\label{V.a.13}
\Delta\hat{\la}_\infty=c_0\ \delta_0 \quad ,
\ee
and $c_0=2\pi \, (\theta_0-1)\in 2\pi\, {\Z}$ where $\theta_0$ is the order of the zero of the pole of $f_\infty$ at $0$.
Consider the 2-sphere $S^2_k$ of radius $e^{-\ov{\la}_k/2}$ tending to infinity and tangent to the 2-plane $P^2_0$ at $0$ and given by $z_i=0$ for $i\ge 3$ at 0 and contained in the half three space $E^3_+$ given
by $z_i=0$ for $i\ge 4$ and $z_3\ge 0$.

\medskip

Let $\pi_k$ be the stereographic projection from $S^2_k$ into the two plane $P^2_0$ such that $\pi_k(0)=(0\cdots 0)$ and $\pi_k(0,0, 2\,e^{-\ov{\la}_k/2},0\cdots 0)=\infty$. 
It is clear that for any $R>0$ and 
\[
\forall\, l\in {\N}\quad\quad\quad\|\pi_k^{-1}(z)-z\|_{C^l(B^2_R(0))}\le \, C_{R,l}\ e^{\ov{\la}_k/2}
\]
where $B^2_R(0):=B_R^m(0)\cap P^2_0$. Hence, for any choice $2\al<\delta<1$,  we have 
\be
\label{V.a.14}
\lim_{k\rightarrow +\infty}\|\pi_k^{-1}\circ\hat{\vec{\Phi}}_\infty-\hat{\vec{\Phi}}_\infty \|_{C^l(B_\delta(0)\setminus B_{2\al}(0))}\rightarrow 0 \quad.
\ee
The advantage of considering $\pi_k^{-1}\circ\hat{\vec{\Phi}}_\infty$ instead of $\hat{\vec{\Phi}}_\infty $ is that $\pi_k^{-1}\circ\hat{\vec{\Phi}}_\infty(D^2)$ is compact
even if $\theta_0<0$ and since $\pi_k^{-1}$ is conformal $\pi_k^{-1}\circ\hat{\vec{\Phi}}_\infty$ is also conformal.
Using (\ref{V.a.11}) we have for any choice $2\al<\delta<1$
\be
\label{V.a.15}
\hat{\vec{\Phi}}_k-\pi_k^{-1}\circ\hat{\vec{\Phi}}_\infty\rightharpoonup 0\quad\quad\mbox{ weakly in }(W^{1,\infty})^\ast\cap W^{2,2}(B_\delta(0)\setminus B_{2\al}(0))\quad.
\ee
Using Fubini's Theorem together with the mean value formula, we can   find a ``good radius'' $r_k\in (1/2,1)$ such that
\be
\label{V.a.16}
\limsup_{k\rightarrow+\infty}\int_{\p B_{r_k}}|\nabla^2 \hat{\vec{\Phi}}_k|^2\ dl<+\infty \quad,
\ee
where $dl$ is the length form on $\p B_{r_k}$. Because of (\ref{V.a.15}), we have that
\be
\label{V.a.17}
\hat{\vec{\Phi}}_k-\pi_k^{-1}\circ\hat{\vec{\Phi}}_\infty\rightharpoonup 0\quad\quad\mbox{ weakly in }H^{3/2}({\p B_{r_k}})\quad
\ee
and 
\be
\label{V.a.18}
\p_r\hat{\vec{\Phi}}_k-\p_r(\pi_k^{-1}\circ\hat{\vec{\Phi}}_\infty)\rightharpoonup 0\quad\quad\mbox{ weakly in }H^{1/2}({\p B_{r_k}})\quad.
\ee
Combining (\ref{V.a.16})...(\ref{V.a.18}) we deduce
\be
\label{V.a.19}
\hat{\vec{\Phi}}_k-\pi_k^{-1}\circ\hat{\vec{\Phi}}_\infty\rightharpoonup 0\quad\quad\mbox{ weakly in }W^{2,2}({\p B_{r_k}})\quad
\ee
and 
\be
\label{V.a.20}
\p_r\hat{\vec{\Phi}}_k-\p_r(\pi_k^{-1}\circ\hat{\vec{\Phi}}_\infty)\rightharpoonup 0\quad\quad\mbox{ weakly in }W^{1,2}({\p B_{r_k}})\quad.
\ee
Consider the map solving
\be
\label{V.a.21}
\lf\{
\begin{array}{l}
\ds\Delta^2\ti{\vec{\Phi}}_k=0\quad\quad\mbox{ in }B_{r_k}\setminus B_{r_k/2}\quad\\[5mm]
\ds\ti{\vec{\Phi}}_k=\hat{\vec{\Phi}}_k\quad\quad\mbox{ in }\p B_{r_k}\\[5mm]
\ds\ti{\vec{\Phi}}_k=\pi_k^{-1}\circ\hat{\vec{\Phi}}_\infty\quad\quad\mbox{ in }\p B_{r_k/2}\\[5mm]
\ds\p_r\ti{\vec{\Phi}}_k=\p_r\hat{\vec{\Phi}}_k\quad\quad\mbox{ in }\p B_{r_k}\\[5mm]
\ds\p_r\ti{\vec{\Phi}}_k=\p_r(\pi_k^{-1}\circ\hat{\vec{\Phi}}_\infty)\quad\quad\mbox{ in }\p B_{r_k/2} \quad.
\end{array}
\rg.
\ee
Since $\hat{\vec{\Phi}}_\infty$ is holomorphic, and hence biharmonic,  combining (\ref{V.a.14}), (\ref{V.a.19})...(\ref{V.a.21}) together with classical
elliptic estimates we obtain
\be
\label{V.a.22}
\ti{\vec{\Phi}}_k-\hat{\vec{\Phi}}_\infty\rightharpoonup 0\quad\quad\mbox{ weakly in }W^{5/2,2}(B_{r_k}\setminus B_{r_k/2}) \quad,
\ee
which implies, by Sobolev embeddings,
\be
\label{V.a.23}
\lim_{k\rightarrow 0}\|\ti{\vec{\Phi}}_k-\hat{\vec{\Phi}}_\infty\|_{C^{1,\al}(B_{r_k}\setminus B_{r_k/2})}=0\quad,
\ee
for any $\al<1/2$.

\medskip

We extend $\ti{\vec{\Phi}}_k$ by $\pi_k^{-1}\circ\hat{\vec{\Phi}}_\infty$ in $B_{r_k/2}$ and by $\hat{\Phi}_k$ in the complement of $B_{r_k}$. Finally we denote
\[
\vec{\zeta}_k:=e^{\ov{\la}_k}\ \ti{\vec{\Phi}}_k(x/s_k)+\vec{\Phi}_k(0,1/2)
\]
in such a way that
\begin{itemize}
\item[i)] 
\be
\label{V.b.23}
\vec{\zeta}_k\equiv\vec{\Phi}_k\quad\quad\mbox{ in } {\C}\setminus B_{s_k} \quad,
\ee
\item[ii)]
\be
\label{V.c.23}
\lim_{k\rightarrow 0}diam(\vec{\zeta}_k(B_{s_k}(0)))=0 \quad,
\ee
\item[iii)]
\be
\label{V.d.23}
\lim_{k\rightarrow +\infty}\int_{B_{{s_k}}(0)}\lf[1+|{\mathbb I}^0_{\vec{\zeta}_k}|^2\rg]\ dvol_{g_{\vec{\zeta}_k}}=0\quad,
\ee
\item[iv)]
and  finally  introducing
\be
\label{V.a.24}
\sigma_{k}:=\frac{g^{11}_{\vec{\zeta}_k}-g^{22}_{\vec{\zeta}_k}-2\,i\, g^{12}_{\vec{\zeta}_k}}{g^{11}_{\vec{\zeta}_k}+g^{22}_{\vec{\zeta}_k}+\sqrt{ det\, g_{\vec{\zeta}_k}}}
\ee
where $g^{ij}_{\vec{\zeta}_k}:=\p_{x_i}\vec{\zeta}_k \cdot\p_{x_j}\vec{\zeta}_k$, using  (\ref{V.a.22})  and (\ref{V.a.23}), we have that
\be
\label{V.a.25}
supp(\sigma_{k})\subset B_{s_k\,r_k}\setminus B_{s_k\,r_k/2}\quad\quad\mbox{ and }\quad\quad\lim_{k\rightarrow +\infty}\|\sigma_k\|_{L^\infty({\C})}+\|\nabla \sigma_k\|_{L^2({\C})}=0\quad.
\ee
\end{itemize}
Let $\psi_k(z):=\phi_k(z)+z$ where $\phi_k$ is the fixed point in $\dot{H}^1({\C})$ given by
\[
\lf\{
\begin{array}{l}
\ds\phi_k(z)+\frac{1}{\pi z}\ast(\sigma_k\ \p_z\phi_k)=-\frac{1}{\pi z}\ast\sigma_k\\[5mm]
\ds\quad\phi_k(0)=0\quad.
\end{array}
\rg.
\]
Notice that $\phi_k$ satisfies then
\be
\label{V.a.26}
\p_{\ov{z}}\phi_k=\sigma_k\ \p_z \phi_k+\sigma_k\quad.
\ee
and
\be
\label{V.a.26a}
\|\nabla \phi_k\|_{L^2({\C})}\le C\ \|\sigma_k\|_{L^2({\C})}\rightarrow 0 \quad.
\ee
It is a classical fact from quasi-conformal mapping theory and from the classical analysis of Beltrami equation that $\psi_k$ realizes an H\"older homeomorphism from
${\C}\cup\{\infty\}$ into ${\C}\cup\{\infty\}$ (see for instance Section 4.2 of \cite{IT}), moreover there exists $p>2$ such that
\be
\label{V.a.27}
\|\nabla \phi_k\|_{L^p({\C})}\le C\ \|\sigma_k\|_{L^p({\C})}\rightarrow 0 \quad.
\ee
This implies in particular that
\be
\label{V.a.28}
\|\psi_k(x)-x\|_{C^{0,\al}}\rightarrow 0
\ee
for some $0<\al<1$ and thus
\be
\label{V.a.29}
\psi_k\lf(B_{s_k\,r_k}(0)\setminus B_{s_k\,r_k/2}(0)\rg)\subset B_{2 (s_k\,r_k)^\al}(0)
\ee
for $k$ large enough.
A classical computation gives
\be
\label{V.a.30}
\sigma_{\vec{\zeta}_k\circ\psi_k^{-1}}=\frac{g^{11}_{\vec{\zeta}_k\circ\psi_k^{-1}}-g^{22}_{\vec{\zeta}_k\circ\psi_k^{-1}}-2\,i\, g^{12}_{\vec{\zeta}_k\circ\psi_k^{-1}}}{g^{11}_{\vec{\zeta}_k\circ\psi_k^{-1}}+g^{22}_{\vec{\zeta}_k\circ\psi_k^{-1}}+\sqrt{ det\, g_{\vec{\zeta}_k\circ\psi_k^{-1}}}}=0
\ee
where $g^{ij}_{\vec{\zeta}_k\circ\psi_k^{-1}}:=\p_{x_i}(\vec{\zeta}_k\circ\psi_k^{-1}) \cdot\p_{x_j}(\vec{\zeta}_k\circ\psi_k^{-1})$.
Again classical computations (see Section 4.2 of \cite{IT}) gives that $\psi_k^{-1}$ is the \underbar{normal solution} to
 \[
 \p_{\ov{w}}\psi_k^{-1}=-\sigma_k\circ\psi_k^{-1}\ \p_{w}\ \psi_k^{-1}
 \]
where $w=y_1+iy_2\in{\C}$ (i.e. $\psi_k$ is the unique continuous function satisfying the above identity and such that $\psi_k(0)=0$ and such that $\p_{w}\psi_k-1$ belongs to $L^p(\C)$, for some $p>2$; for more details see \cite[Theorem 4.24]{IT}). Moreover there exists $\varphi_k$ such that $\psi_k^{-1}(y)=y+\varphi_k(y)$ and, because of (\ref{V.a.29}), we have that
\[
Supp(-\sigma_k\circ\psi_k^{-1})\subset B_{2 (s_k\,r_k)^\al}(0)\quad\quad\mbox{ and }\quad\quad\lim_{k\rightarrow +\infty}\|-\sigma_k\circ\psi_k^{-1}\|_{L^\infty({\C})}=0\quad.
\]
Thus, as before, we have the existence of $p>2$ such that
\be
\label{V.a.31}
\|\nabla \varphi_k\|_{L^2({\C})}+\|\nabla \varphi_k\|_{L^p({\C})}\rightarrow 0 \quad.
\ee 
We now go back to the sphere $\Sp^2$ and we set
\be
\label{V.a.31.b}
\lf\{
\begin{array}{l}
\ds\vec{\xi}_k:=\vec{\zeta}_k\circ\psi_k^{-1}\circ\pi\\[5mm]
\ds \Psi_k:=\pi^{-1}\circ\psi_k\circ\pi \quad,
\end{array}
\rg.
\ee
where we recall that $\pi$ is the stereographic projection from $\Sp^2$ into ${\C}$ that sends $a\in \Sp^2$ into $0$. Because of (\ref{V.a.28}), $\Psi_k$ converges uniformly to the identity
in any compact of $\Sp^2\setminus\{-a\}$. Moreover, still because of (\ref{V.a.28}), for any $\ep>0$ there exists $k_0$ and $\delta$ such that for all $k>k_0$ and $r<\delta$
$\|\Psi_k(x) +a\|_{L^\infty(B_r(-a))}<\ep$. Thus $\Psi_k$ converges uniformly to the identity on $\Sp^2$. Finally combining this fact with (\ref{V.b.23}), (\ref{V.c.23}) and (\ref{V.d.23})
we have proved the ``cutting and filling Lemma''~\ref{lm-V.2}.\hfill $\Box$

\medskip

\noindent{\bf Proof of Lemma~\ref{lm-V.1}.} Let $\{\vec{\Phi}_k\}_{k \in \N}$ be a sequence of conformal, possibly branched, weak immersions in ${\mathcal F}_{\Sp^2}$ satisfying the assumption of the lemma. Denote by $\pi$ the stereographic projection that sends $a$ to zero and which is almost an isometry from small geodesic balls centered at $a$ into the corresponding euclidian ball
centered at $0$. We have for instance that $\pi(B_{r^i_k}(a^i_k))$ (resp.  $\pi(B_{\ep_k^{-1}r^i_k}(a^i_k))$ is ``almost'' the  ball of center $x_k^i:=\pi(a_k^i)$ and radius
$r_k^i$ (resp. radius $\ep_k^{-1} r_k^i$). In oder to simplify the presentation we will identify $\pi(B_{r^i_k}(a^i_k))$ with $B^2_{r_k^i}(x_k^i)$, $\pi(B_{\ep_k^{-1}r^i_k}(a^i_k))$
with $B^2_{\ep_k^{-1}r_k^i}(x_k^i)$, $\pi(B_{\delta_k}(a))$ with $B^2_{\delta_k}(0)$ and $\pi(B_{\ep_k\, \delta_k}(a))$ with $B^2_{\ep_k\,\delta_k}(0)$ in the list of assumptions 
going from (\ref{V.a.b.1}) until (\ref{V.a.e.1}). 

Let $D_k(x):=x/\ep_k\delta_k$ be the dilation in ${\C}$ of factor $(\ep_k\delta_k)^{-1}$. We consider the new conformal immersion $\ti{\vec{\Phi}}_k$ given by
\[
\ti{\vec{\Phi}}_k:=\vec{\Phi}_k\circ\pi^{-1}\circ D_k\circ \pi\quad.
\]
We can apply Lemma~\ref{lm-V.2} $N+1$ times respectively in the ball $\pi^{-1}({\C}\setminus B_{1/\sqrt{\ep_k}}(0):=B_{s_k}(-a)$ and in the balls 
$\pi^{-1}(B_{r^i_k/(\ep_k^2\delta_k)}(x_k^i/(\ep_k\delta_k))):=B_{r^i_k/(\ep_k^2\delta_k)}(c_k^i)$ where
 $$
\rho_i^k:=\frac{r^i_k}{\ep_k^2\delta_k}\rightarrow 0 \quad\quad\mbox{ and }\quad\quad c_k^i:=\pi^{-1}(x_k^i/(\ep_k\delta_k))\rightarrow a \quad .
 $$
We then generate a sequence $\{\Psi_k\}_{k\in \N}$  of bilipschitz quasi conformal homeomorphisms of $\Sp^2$ converging uniformly to the identity (with a distortion $\mu_{\Psi_k}$ converging
uniformly to zero) and a sequence $\vec{\xi}_k$ of conformal, possibly branched, weak immersions in ${\mathcal F}_{\Sp^2}$ satisfying
\be
\label{V.a.32}
\vec{\xi}_k\circ\Psi_k=\ti{\vec{\Phi}}_k\quad\quad\mbox{ in }\Sp^2\setminus B_{s_k}(-a)\bigcup_{i=1}^NB_{\rho_k^i}(c_k^i) \quad .
\ee
Moreover
\be
\label{V.a.33}
\lim_{k\rightarrow +\infty}diam\lf(\vec{\xi}_k\circ\Psi_k(B_{s_k}(-a))\rg)+\sum_{i=1}^Ndiam\lf(\vec{\xi}_k\circ\Psi_k(B_{\rho_k^i}(c_k^i))\rg)= 0 \quad,
\ee
whereas, from the assumption (\ref{V.a.h.1}) one has
\be
\label{V.a.34}
\liminf_{k\rightarrow+\infty} diam\lf(\vec{\xi}_k\circ\Psi_k\lf(\Sp^2\setminus \lf(B_{s_k}(-a)\bigcup_{i=1}^NB_{\rho_k^i}(c_k^i)\rg)\rg)\rg)>0 \quad.
\ee
We apply  the ``good-gauge extraction Lemma''~\ref{lm-IV.1} to $\vec{\xi}_k$ and we obtain $u_k\in {\mathcal M}^+(\Sp^2)$, an element $\vec{\zeta}_\infty\in{\mathcal F}_{\Sp^2}$
and $Q$ points $b_1,\cdots, b_Q$ such that
\be
\label{V.a.35}
\vec{\zeta}_k:=\vec{\xi}_k\circ u_k\rightharpoonup \vec{\zeta}_\infty \quad\quad\mbox{weakly in }W^{2,2}(\Sp^2\setminus\{b_1,\cdots, b_Q\})\quad.
\ee
Because of (\ref{V.a.33}) there exists $p_0,p_1\cdots p_N\in M^m$ such that
\[
\vec{\xi}_k\circ\Psi_k(B_{s_k}(-a))\rightarrow p_0\quad\quad\mbox{ and }\quad\forall i\in\{1\cdots N\}\quad\vec{\xi}_k\circ\Psi_k(B_{\rho_k^i}(c_k^i))\rightarrow p_i\quad.
\]
Since the Willmore energy of $\vec{\zeta}_\infty$ is finite, each point admits a controlled number of preimages. Denote by
\[
\{x_1\cdots x_L\}=\bigcup_{i=0}^N\vec{\zeta}_\infty^{-1}(\{p_i\})\quad.
\]
For any $\ep>0$ and for $k$ large enough 
$$
\zeta_k\lf(\Sp^2\setminus \bigcup_{l=1}^L B_\ep(x_l)\rg)\cap \lf(\vec{\xi}_k\circ\Psi_k(B_{s_k}(-a))\bigcup_{i=1}^N\vec{\xi}_k\circ\Psi_k(B_{\rho_k^i}(c_k^i))\rg)=\emptyset \quad.
$$
Thus for any $\ep$ and $k$ large enough
\be
\label{V.a.36}
\forall\, x\in \Sp^2\setminus \bigcup_{l=1}^L B_\ep(x_l)\quad\quad\vec{\zeta}_k(x)=\ti{\vec{\Phi}}_k\circ\Psi_k^{-1}\circ u_k(x)\quad.
\ee
Consider the following sequence of quasi-conformal homeomorphism of the sphere
\[
\La_k:=\Psi_k^{-1}\circ u_k\quad,
\]
the identity \eqref{V.a.36} implies that $\La_k$ is conformal on $\Sp^2\setminus \bigcup_{l=1}^L B_\ep(x_l)$ for $k$ large enough. Let $\pi$ be a fixed stereographic projection that sends 
none of the points $x_l$ to infinity : for $\ep$ small enough $\pi( \bigcup_{l=1}^L B_\ep(x_l))$ is sent into a fixed ball $B_R(0)$ of the complex plane. Let
\[
\mu_{\La_k\circ\pi^{-1}}:=\frac{\p_{\ov{z}}(\La_k\circ\pi^{-1})}{\p_{{z}}(\La_k\circ\pi^{-1})} \quad.
\]
We have then from the previous consideration and the proof of Lemma~\ref{lm-V.2}
\be
\label{V.a.38}
supp(\mu_{\La_k\circ\pi^{-1}})\subset B_R(0)\quad\quad\mbox{ and }\quad\quad\|\mu_{\La_k\circ\pi^{-1}}\|_{L^\infty({\C})}\longrightarrow 0 \quad.
\ee
Let ${\psi}_k$ be the normal solution (see again \cite{IT} chapter 4) of
\[
\lf\{
\begin{array}{l}
\ds\p_{\ov{z}}\psi_k=\mu_{\La_k\circ\pi^{-1}}\ \p_{z}\psi_k\quad\quad\mbox{ on }{\C}\\[5mm]
\ds \psi_k(0)=0
\end{array}
\rg.
\]
such that there exists $p>2$ for which $\nabla(\psi_k(x)-x)\in L^p({\C})$. It is a classical result, since (\ref{V.a.38}) holds, that $\{\psi_k\}_{k\in \N}$ and $\{\psi_k^{-1}\}_{k\in \N}$ are compact in $C^{0,\al}$
for some $\al<1$ (see \cite{IT} 4.2.3), and converge to the identity. Moreover, a classical computation (see \cite[Proposition 4.13]{IT}) gives that $\La_k\circ\pi^{-1}\circ\psi_k^{-1}\circ\pi$ is conformal. 

\medskip

Consider now $\zeta_k\circ\pi^{-1}\circ\psi_k^{-1}\circ\pi$. This sequence of maps converge in $C^{0,\al}$ norm to $\vec{\zeta}_\infty$ and in particular we have
\be
\label{V.a.39}
\ti{\vec{\Phi}}_k\circ\La_k\circ \pi^{-1}\circ\psi_k^{-1}\circ\pi\longrightarrow \vec{\zeta}_\infty \quad\quad\mbox{ in }C^{0,\al}\lf(\Sp^2\setminus \bigcup_{l=1}^L B_\ep(x_l)\rg)\quad.
\ee
We apply  now Proposition~\ref{Prop:ConcComp} to $\ti{\vec{\Phi}}_k\circ\La_k\circ \pi^{-1}\circ\psi_k^{-1}\circ\pi$ which is conformal on $\Sp^2$. Assume that, modulo extraction of a subsequence, $\ti{\vec{\Phi}}_k\circ\La_k\circ \pi^{-1}\circ\psi_k^{-1}\circ\pi $ would converge
uniformly to a point on any compact $K\subset \Sp^2\setminus\{a_1\cdots a_P\}$ for some finite collection of points $\{a_1\cdots a_P\}$; this would contradict (\ref{V.a.39}).
Thus we can take
\[
f_k:=\pi^{-1}\circ D_k\circ\pi\circ\La_k\circ \pi^{-1}\circ\psi_k^{-1}\circ\pi\quad,
\]
and one easily checks that the required conditions (\ref{V.a.f.1})...(\ref{V.a.k.2}) are fullfield for this choice of $f_k$ and Lemma~\ref{lm-V.1} is proved.\hfill $\Box$

\section{Domain decomposition and the proof of Theorem~\ref{th-I.1}.}
\reset

Before  moving to the proof of Theorem~\ref{th-I.1} we will first establish the following lemma.

\begin{Lm} {\bf[Domain decomposition lemma]}
\label{lm-VI.1}
Let $\{\vec{\Phi}_k\}_{k \in \N}{\subset \mathcal F}_\Sp^2$ be a sequence of conformal weak, possibly branched, immersion  into $M^m$. Assume that
\be
\label{VI.1}
\limsup_{k\rightarrow +\infty}\int_{\Sp^2}\lf[1+|D\vec{n}_{\vec{\Phi}_k}|_h^2\rg]\ dvol_{g_k}<+\infty\quad,
\ee
where $dvol_{g_k}$ denotes the volume form associated to the induced metric $g_k:=\vec{\Phi}_k^\ast h$ by $\vec{\Phi}_k$ on 
$\Sp^2$ and $\vec{H}_{\vec{\Phi}_k}$ is the mean curvature vector associated to the immersion $\vec{\Phi}_k$. Then, modulo extraction of a subsequence, there
exists $N\in {\N}$, there exist $N$ sequences
of M\"obius transformations in ${\mathcal M}^+(\Sp^2)$, $(f_k^i)_{i=1\cdots N}$, $N$ elements of ${\mathcal F}_{\Sp^2}$, $(\vec{\xi}^i_\infty)_{i=1\cdots N}$, and $N$ natural integers 
$(N_i)_{i=1\cdots N}$ such that
\be
\label{VI.2}
\vec{\Phi}_k\circ f^i_k\rightharpoonup \vec{\xi}^i_\infty\quad\quad\mbox{ weakly in }W^{2,2}(\Sp^2\setminus\{b^{i,1}\cdots b^{i,N^i}\})
\ee
where, for each $i\in\{1\cdots N\}$, $(b^{i,j})_{j=1\cdots N^i}$ is a finite family of points in $\Sp^2$. There exists a sequence $s_k\rightarrow 0$ such that for any $i=1\cdots N$
 \be
 \label{VI.3}
 \vec{\Phi}_k\circ f^i_k\longrightarrow \vec{\xi}^i_\infty\quad\quad\mbox{ uniformly in }C^0\lf(\Sp^2\setminus\bigcup_{j=1}^{N^i} B_{s_k}(b^{i,j})\rg)\quad,
 \ee
 and
 \be
 \label{VI.3a}
 \int_{\Sp^2\setminus \bigcup_{j=1}^{N^i} B_{s_k}(b^{i,j})}1\ dvol_{g_{\vec{\Phi}_k\circ f^i_k}}\longrightarrow\int_{\Sp^2}1\ dvol_{g_{\vec{\xi}^i_\infty}}=Area(\vec{\xi}^i_\infty(\Sp^2))\quad.
 \ee
 Denote
 \be
 \label{VI.4}
 S^i_k:=\Sp^2\setminus\bigcup_{j=1}^{N^i} B_{s_k}(b^{i,j})\quad .
 \ee
For any $i'\ne j$ there exists $i\in \{1\cdots N^i\}$ such that
 \be
 \label{VI.5}
 \lf(f^i_k\rg)^{-1}\circ f^{i'}_k\lf(S^{i'}_k\rg)\subset B_{s_k}(b^{i,j})\quad.
 \ee
 For each $i\in\{1\,cdots, N\}$ and for each $j\in\{1,\cdots, N^i\}$ we denote by $J^{i,j}$ the set of indices $i'$ such that (\ref{VI.5}) holds. We have
 finally
 \be
 \label{VI.6}
 \begin{array}{l}
\ds\forall\, i\in\{1,\cdots, N\}\quad,\quad\forall j\in\{1,\cdots, N^i\}\\[5mm]
 \ds\quad\quad\lim_{k\rightarrow \infty}diam\lf(\vec{\Phi}_k\circ f^i_k\lf(B_{s_k}(b^{i,j})\setminus\bigcup_{i'\in J^{i,j}}Conv\lf( \lf(f^i_k\rg)^{-1}\circ f^{i'}_k\lf(S^{i'}_k\rg)\rg)\rg)\rg)=0\quad,
 \end{array}
 \ee
where $Conv(X)$ denotes the convex hull of a set $X$, and
 \be
 \label{VI.6b}
 \begin{array}{l}
\ds\forall\, i\in\{1\cdots N\}\quad,\quad\forall j\in\{1\cdots N^i\}\\[5mm]
 \ds\quad\quad\lim_{k\rightarrow \infty}Area\lf(\vec{\Phi}_k\circ f^i_k\lf(B_{s_k}(b^{i,j})\setminus\bigcup_{i'\in J^{i,j}}Conv\lf( \lf(f^i_k\rg)^{-1}\circ f^{i'}_k\lf(S^{i'}_k\rg)\rg)\rg)\rg)=0\quad.
 \end{array}
 \ee
\hfill $\Box$
\end{Lm}
\noindent{\bf Proof of Lemma~\ref{lm-VI.1}.} We construct the $f^i_k$'s, the $\vec{\xi}_i^\infty$'s and the $b^{i,j}$'s by induction. We first apply the ``good gauge extraction'' Lemma~\ref{lm-IV.1}
and we generate a first sequence $f_k^1$ of elements of ${\mathcal M}^+(\Sp^2)$ and a first element $\vec{\xi}_\infty^1\in {\mathcal F}_{\Sp^2}$ as well as a first collection of points
$\{b^{1,1}\cdots b^{1,N^1}\}$ such that
\[
\vec{\xi}_k^1:=\vec{\Phi}_k\circ f^1_k\rightharpoonup \vec{\xi}^1_\infty\quad\quad\mbox{ weakly in }W^{2,2}(\Sp^2\setminus\{b^{1,1}\cdots b^{1,N^1}\})\quad.
\]
Since the weak convergence in $ W^{2,2}(\Sp^2\setminus\{b^{1,1}\cdots b^{1,N^1}\})$ 
implies a strong $C^0$ convergence of $\{\vec{\xi}_k^1\}_{k\in \N}$ on  any compact  of $S^1_\infty:=\Sp^2\setminus\{b^{1,1}\cdots b^{1,N^1}\}$
but also a strong $L^2$ convergence of the conformal factors $\{e^{\la_k^1}\}_{k\in \N}$  satisfying  $g_{\vec{\xi}_k^1}= e^{2\,\la_k^1}\ g_{\Sp^2}$, and moreover since the limit $\vec{\xi}_\infty^1\in {\mathcal F}_{\Sp^2}$ is lipschitz all over $\Sp^2$,
one can always find $t_k\rightarrow 0$ such that
\be
 \label{VI.7}
 \vec{\xi}_k^1:=\vec{\Phi}_k\circ f^1_k\longrightarrow \vec{\xi}^1_\infty\quad\quad\mbox{ uniformly in }C^0\lf(\Sp^2\setminus\bigcup_{j=1}^{N^1} B_{t_k}(b^{1,j})\rg)\quad,
 \ee
 and such that
 \be
 \label{VI.7ab}
 \int_{\Sp^2\setminus \bigcup_{j=1}^{N^i} B_{t_k}(b^{i,j})}1\ dvol_{g_{\vec{\Phi}_k\circ f^1_k}}\longrightarrow\int_{\Sp^2}1\ dvol_{g_{\vec{\xi}^1_\infty}}=Area(\vec{\xi}^1_\infty(\Sp^2))\quad .
 \ee
Consider now for any $j\in\{1\cdots N^1\}$ and for $l\in{\N}$ the annuli
$$
A^{j,l}_k:= B_{t_k^{\frac{1}{l+1}}}(b^{1,j})\setminus B_{t_k^{\frac{1}{l}}}(b^{1,j}) \quad.
$$
Because of (\ref{VI.1}), one can find $l_k\in {\N}$ such that $t^{\frac{1}{l_k+1}}\rightarrow 0$ and 
\[
\ep_k^1:={t_k^{\frac{1}{l_k}-\frac{1}{l_k+1}}}\longrightarrow 0\quad\quad\mbox{ and }\quad\forall\, j\quad\lim_{k\rightarrow +\infty}\int_{A^{j,l_k}}\lf[1+|\vec{H}_{\vec{\Phi}_k}|_h^2\rg]\ dvol_{g_k}=0 \quad.
\]
We then choose $\delta_k^1:=t^{\frac{1}{l^j_k+1}}$.

\medskip

In case for all $j=1\cdots N^1$ one has
\be
\label{VI.8}
\lim_{k\rightarrow 0}diam\lf(\vec{\Phi}_k\circ f^1_k\lf(B_{\delta_k^1}(b^{1,j})\rg)\rg)=0 \quad,
\ee
then we stop the procedure at this stage. If not, we take each of the $B_{\delta_k^1}(b^{1,j})$ such that (\ref{VI.8}) does not happen and for each of them we apply the ``diameter tracking procedure''
lemma~\ref{lm-V.1} with $a:=b^{1,j}$, $\delta_k:= \delta_k^1$ and $\ep_k:=\ep_k^1$ and no $a_k^i$ first. We generate then new $f^i_k$ and observe that each such generation costs
an amount of  energy which is at least
\[
\inf_{\vec{\xi}\in {\mathcal F}_{\Sp^2}}L(\vec{\xi})=c_0>0 \quad .
\]
Thus the procedure must stop after finitely many iterations and Lemma~\ref{lm-VI.1} will be proved once (\ref{VI.6b}) will be established.

\medskip

\noindent{\it Proof of (\ref{VI.6b}).} There exists a sequence of conformal transformation $\varphi_k$ in ${\mathcal M}^+(\Sp^2)$  such that
\[
\varphi_k^{-1}\lf(B_{s_k}(b^{i,j})\setminus\bigcup_{i'\in J^{i,j}}Conv\lf( \lf(f^i_k\rg)^{-1}\circ f^{i'}_k\lf(S^{i'}_k\rg)\rg)\rg)= \Sp^2\setminus\bigcup_{l=1}^{N^{i,j}+1} B_{s_k^l}(a_k^l)
\]
and there exists $\ep_k\rightarrow 0$ s.t. $\ep_k^{-1}\, s_k^l\rightarrow 0$ for any $l$ and
\[
\lim_{k\rightarrow +\infty}\int_{B_{\ep_k^{-1}\,s_k^l}(a_k^l)\setminus B_{s_k^l}(a_k^l)}\lf[1+|{\mathbb I}_{\vec{\Phi}_k\circ f^i_k\circ\varphi_k}|^2\rg]\ dvol_{g_{\vec{\Phi}_k\circ f^i_k\circ\varphi_k}}=0 \quad.
\]
To each $l$ we apply the  {\it cutting and filling Lemma~\ref{lm-V.2}} for $a=a^l_k$, $s_k=\ep_k^{-1}\,s_k^l$ and $t_k=s_k^l$. We obtain a new conformal immersion
$\vec{\zeta}_k^i$ such that 
\be
\label{n.VI.1}
d_k^i:=diam(\vec{\zeta}_k^i(\Sp^2))\rightarrow 0\quad\quad,\quad\quad\limsup_{k\rightarrow +\infty} G(\vec{\zeta}_k^i)<+\infty
\ee
and such that for all $k$
\be
\label{n.VI.2}
\vec{\Phi}_k\circ f_k^i\lf(B_{s_k}(b^{i,j})\setminus\bigcup_{i'\in J^{i,j}}Conv\lf( \lf(f^i_k\rg)^{-1}\circ f^{i'}_k\lf(S^{i'}_k\rg)\rg)\rg)\subset\vec{\zeta}_k^i(\Sp^2)\quad.
\ee
By  viewing $M^m$ as being isometrically embedded in a euclidian space ${\R}^n$, (\ref{n.VI.1}) together with (\ref{xy1}) imply that there exists
$p_k\in{\R}^n$ such that
\be
\label{n.VI.3}
\limsup_{k\rightarrow +\infty}\int_{\Sp^2}\lf[|D\vec{n}_{\vec{I}\circ\vec{\zeta}_k}|^2\rg]\ dvol_{\vec{I}\circ\vec{\zeta}_k}<+\infty\quad\mbox{ and }\quad\vec{I}\circ\vec{\zeta}_k^i(\Sp^2)\subset B_{d_k}(p_k)\quad.
\ee
By using L. Simon's monotonicity formula \cite{Si}, we deduce that
\be
\label{n.VI.4}
\limsup_{k\rightarrow +\infty} d_k^{-2}\ Area\lf(\vec{\zeta}_k^i(\Sp^2)\rg)=\limsup_{k\rightarrow +\infty} d_k^{-2}\ Area\lf(\vec{\zeta}_k^i(\Sp^2)\cap B_{d_k}(p_k)\rg)<+\infty\quad.
\ee
Thus we infer that
\be
\label{n.VI.5}
Area\lf(\vec{\Phi}_k\circ f_k^i\lf(B_{s_k}(b^{i,j})\setminus\bigcup_{i'\in J^{i,j}}Conv\lf( \lf(f^i_k\rg)^{-1}\circ f^{i'}_k\lf(S^{i'}_k\rg)\rg)\rg)\rg)\le Area\lf(\vec{\zeta}_k^i(\Sp^2)\rg)=O(d_k^2)\rightarrow 0\quad.
\ee
This implies  (\ref{VI.6b})  and  Lemma~\ref{lm-VI.1} is finally proved.\hfill $\Box$

\medskip

\noindent{\bf Proof of Theorem~\ref{th-I.2}.} We construct $\Psi_k$ step by step in the following way.

We consider in $\Sp^2$ $N^1$ disjoint fixed balls $(B^{1,j})_{j=1\cdots N^1}$ and we consider a sequence of diffeomorphisms $\La_k^1$
\[
\La_k^1\ :\ \Sp^2\setminus\bigcup_{j=1}^{N^1} B_{s_k}(b^{1,j})\longrightarrow \Sp^2\setminus \bigcup_{j=1}^{N^1}B^{1,j}\quad,
\]
such that
\[
\La_k^1 (\p B_{s_k}(b^{1,j}))=\p B^{1,j}\quad,
\]
and
\be
\label{VI.9}
\limsup_{k\rightarrow +\infty}\|\nabla(\La_k^1)^{-1}\|_{L^\infty(\Sp^2\setminus \bigcup_{j=1}^{N^1}B^{1,j})}<+\infty\quad,
\ee
and such that $\Xi_k^1:=(\La_k^1)^{-1}$ converges weakly in $(W^{1,\infty})^\ast$ to a limiting diffeomorphism 

$$\Xi_\infty^1\ :\ \Sp^2\setminus \bigcup_{j=1}^{N^1}\ov{B^{1,j}}\longrightarrow \Sp^2\setminus\{b^{1,1}\cdots b^{1,N^1}\}\quad.$$

\medskip

We adopt the same notation as in the proof of Lemma~\ref{lm-VI.1} : for any $k\in {\N}$ and $i=1\cdots N$
\[
S^i_k:=\Sp^2\setminus\bigcup_{j=1}^{N^i} B_{s_k}(b^{i,j})\quad,
\]
and $J^{i,j}$ is the following set of indices
\[
J^{i,j}:=\lf\{i'\quad:\quad  \lf(f^i_k\rg)^{-1}\circ f^{i'}_k\lf(S^{i'}_k\rg)\subset B_{s_k}(b^{i,j})\rg\} \quad;
\]
we also let $\hat{J}^{i,j}$  be the following subset of $J^{i,j}$
\[
\hat{J}^{i,j}:=\lf\{i'\in J^{i,j}\quad:\quad\nexists\,i''\in J^{i,j}\quad\mbox{ s.t. }(f^i_k)^{-1}\circ f^{i'}_k\lf(S^{i'}_k\rg)\subset Conv\lf((f^i_k)^{-1}\circ f^{i''}_k\lf(S^{i''}_k\rg)\rg)\rg\}\quad,
\]
where $Conv(X)$ is the convex hull of $X$ in $\Sp^2$. We denote by $N^{i,j}$ the cardinality of $\hat{J}^{i,j}$. Recall that, due to (\ref{VI.6}), one has for any $i\in\{1\cdots N\}$ and any $j\in \{1\cdots N^{i,j}\}$
\be
\label{VI.10}
\lim_{k\rightarrow 0}diam\lf(\vec{\Phi}_k\circ f^i_k\lf(B_{s_k}(b^{i,j})\setminus \bigcup_{i'\in\hat{J}^{i,j}}\ Conv\lf((f_k^i)^{-1}\circ f_k^{i'}(S^{i'}_k)\rg)\rg)\rg)=0 \quad.
\ee

\medskip

For any $j\in \{1\cdots N^1\}$, inside $B^{1,j}$, we fix $N^{1,j}$ disjoint balls independent of $k$. We denote each of this ball by $B^i$ for each $i\in \hat{J}^{i,j}$.
We can always reorder the indexation in such a way that 
\[
(f^i_k)^{-1}\circ f^1_k\lf(S^1_k\rg)\subset B_{s_k}(b^{i,N_i})\quad.
\]
In each of the $B^i$
we fix again $N^i-1$ disjoint balls $(B^{i,j})_{j=1\cdots N^i-1}$ each being independent of $k$.  For each of these $i\in\hat{J}^{1,j}$ we pick a sequence of diffeomorphisms 
$\La_k^i$
\[
\La_k^i\ :\ \Sp^2\setminus\bigcup_{j=1}^{N^i} B_{s_k}(b^{i,j})\longrightarrow B^i\setminus \bigcup_{j=1}^{N^i-1}B^{i,j}\quad,
\]
such that
\[
\forall\, j=1\cdots N^i-1\quad\quad\La_k^i (\p B_{s_k}(b^{i,j}))=\p B^{i,j}\quad\mbox{ and }\quad\La_k^i (\p B_{s_k}(b^{i,N^i}))=\p B^i\quad,
\]
and
\be
\label{VI.11}
\limsup_{k\rightarrow +\infty}\|\nabla(\La_k^i)^{-1}\|_{L^\infty(B^i\setminus \bigcup_{j=1}^{N^i-1}B^{i,j})}<+\infty\quad,
\ee
and such that $\Xi^i_k:=(\La_k^i)^{-1}$ converges weakly in $(W^{1,\infty})^\ast$ to a limiting diffeomorphism 

$$\Xi_\infty^i\ :\ \Sp^2\setminus \bigcup_{j=1}^{N^i}\ov{B^{i,j}}\longrightarrow \Sp^2\setminus\{b^{i,1}\cdots b^{i,N^i}\}\quad.$$

We iterate this procedure in a straightforward way until having exhausted all the indices $i\in\{1\cdots N\}$.

\medskip

We now fix the sequence of diffeomorphism $\Psi_k$ in the following way : 
\begin{itemize}
\item[i)]
\[
\Psi_k:=f_k^1\circ\Xi_k^1\quad\quad\mbox{ on }\Sp^2\setminus \bigcup_{j=1}^{N^1}B^{1,j}\quad,
\]
\item[ii)]
\[
\Psi_k:=f_k^i\circ\Xi_k^i\quad\quad\mbox{ on } B^i\setminus \bigcup_{j=1}^{N^i-1}B^{i,j}\quad,
\]
\item[iii)]
$\Psi_k$ is chosen arbitrarily among the diffeomorphisms from $$B^{i,j}\setminus\bigcup_{i'\in \hat{J}^{i,j}} B^{i'}$$
into
$$
f^i_k\lf(B_{s_k}( b^{i,j})\rg)\setminus \bigcup_{i'\in \hat{J}^{i,j}}f^{{i'}}_k(B_{s_k}(b^{i',N^{i'}}))
$$
such that $\Psi_k:=f_k^i\circ\Xi_k^i$ on $\p B^{i,j}$ and $\Psi_k:= f^{i'}_k\circ\Xi_k^{i'}$ on $\p B^{i'}$.
\end{itemize}
Because of (\ref{VI.10}) we have that, modulo extraction of a subsequence, for any $i>1$ and $j=1\cdots N^{i}-1$, there exists a point $z^{i,j}\in M^m$ such that
\be
\label{VI.12}
\vec{\Phi}_k\circ\Psi_k\longrightarrow p^{i,j}\quad\mbox{ uniformly in }C^0\lf(B^{i,j}\setminus\bigcup_{i'\in \hat{J}^{i,j}} B^{i'} \rg)\quad.
\ee
Observe that
\[
\Xi_k^1\rightharpoonup \Xi_\infty^1\quad\quad\mbox{ weakly in }(W^{1,\infty})^\ast(\Sp^2\setminus \bigcup_{j=1}^{N^1}B^{1,j})\quad,
\]
and for for $i\ge 2$
\[
\Xi_k^i\rightharpoonup \Xi_\infty^i\quad\quad\mbox{ weakly in }W^{1,\infty}\lf(B^i\setminus \bigcup_{j=1}^{N^i-1}B^{i,j}\rg) \quad.
\]
Define now $\vec{f}_\infty$ from  $\Sp^2$ into $M^m$  by 
\[
\vec{f}_\infty:=\vec{\xi}_\infty^1\circ\Xi_{\infty}^1\quad\quad\mbox{ on }\quad \Sp^2\setminus \bigcup_{j=1}^{N^1}B^{1,j}\quad,
\]
by
\[
\vec{f}_\infty:=\vec{\xi}_\infty^i\circ\Xi_{\infty}^i\quad\quad\mbox{ on }\quad B^i\setminus \bigcup_{j=1}^{N^i-1}B^{i,j} \quad,
\]
and by
\[
\vec{f}_\infty:\equiv p^{i,j}\quad\quad\mbox{ on }\quad B^{i,j}\setminus\bigcup_{i'\in \hat{J}^{i,j}} B^{i'} \quad.
\]
It is now straightforward to check that the $N+1$-uplet $(\vec{f}_\infty, \vec{\xi}_\infty^1\cdots\vec{\xi}^N_\infty)$ satisfy the conclusions of Theorem~\ref{th-I.2}. This concludes the proof.\hfill $\Box$

\section{Weak Closure of Bubble Trees of Weak immersions.}

As it is, Theorem~\ref{th-I.2} is a weak-semi-closure result and not a weak-closure result {\it per se} in the sense that the
weak limit does not anymore belong to the same class of {\it weak immersions},  ${\mathcal F}_{\Sp^2}$, but is made of a finite family of elements of ${\mathcal F}_{\Sp^2}$
that we will call a {\it bubble tree of weak immersions}.
In order to remedy to this difficulty we are going to define rigorously the class of {\it bubble tree of weak immersions} and prove afterwards
a weak closure result in this class (see Theorem~\ref{th-VII.1}).

\begin{Dfi}
\label{df-VII.1}{\bf[Bubble trees of weak immersions]}
We call a {\it bubble tree of weak immersions} a $N+1$-tuple $\vec{T}:=(\vec{f},\vec{\Phi}^1\cdots\vec{\Phi}^N)$, where $N$ is an arbitrary integer, $\vec{f}\in W^{1,\infty}(\Sp^2,M^m)$
and $\vec{\Phi}^i\in{\mathcal F}_{\Sp^2}$ for $i=1\cdots N$ satisfy the following conditions.
There exists a family of $N$ geodesic balls $B^i\subset \Sp^2$ such that 
\[
\forall \, i\ne i'\quad\mbox{ either }
\ov{B^i}\subset B^{i'}\quad\mbox{ or }\quad\ov{B^{i'}}\subset B^i\quad\mbox{ and }\quad B^1=\Sp^2 \quad.
\]
For any $i\in \{1\cdots N\}$ there exist a natural integer $N^i$ and $N^i$ disjoint open geodesic balls $B^{i,j}$ strictly included in $B^i$ such that,
\[
 \forall\,i'\ne i\quad\mbox{ either }\quad
\ov{B^i}\subset B^{i'}\quad\mbox{ or }\quad\exists\,j\in\{1\cdots N^i\}\quad\mbox{ s.t. }\quad \ov{B^{i'}}\subset  B^{i,j}\quad.
\]
For any $i\in\{1\cdots N\}$ there exists $N^i$ distinct points  $b^{i,1}\cdots b^{i,N^i}$ of $\Sp^2$ and a lipschitz diffeomorphism
\[
\Xi^i\ :\ B^i\setminus \bigcup_{j=1}^{N^i-1}\ov{B^{i,j}}\quad\longrightarrow\quad \Sp^2\setminus \{b^{i,1}\cdots b^{i,N^i}\}
\]
such that $\Xi_i$ extends to a lipschitz map 
\[
\ov{\Xi}_i\ :\ \ov{B^i}\setminus \bigcup_{j=1}^{N^i-1}{B^{i,j}}\longrightarrow \Sp^2\quad\mbox{ such that }\quad\ov{\Xi}_i(\p B^{i,j})=b^{i,j}\quad\mbox{ and }\quad\ov{\Xi}_i(\p B^i)=b^{i,N^i}\quad.
\]
Moreover
\[
\forall\, i=1\cdots N \quad\text{one has}\quad\vec{f}=\vec{\Phi}^i\circ\Xi^i\quad\quad\mbox{ on }\quad B^i\setminus \bigcup_{j=1}^{N^i-1}B^{i,j}
\]
and for any $i\in\{1\cdots N\}$ and any $j\in\{1\cdots N^i\}$ there exists a point $p^{i,j}\in M^m$ such that
\[
\vec{f}:\equiv p^{i,j}\quad\quad\mbox{ on }\quad B^{i,j}\setminus\bigcup_{i'\in \hat{J}^{i,j}} B^{i'} 
\]
where
\[
J^{i,j}:=\{i'\quad: \quad \ov{B^{i'}}\subset  B^{i,j}\}\quad.
\]
We denote by ${\mathcal T}$ the space of bubble trees of weak immersions. and for any $\vec{T}\in {\mathcal T}$  we denote
\[
G(\vec{T}):=\sum_{i=1}^NG(\vec{\Phi}^i):=\sum_{i=1}^NA(\vec{\Phi}^i)+F(\vec{\Phi}^i)=\sum_{i=1}^N\int_{\Sp^2}\lf[1+\frac{|D\vec{n}_{\vec{\Phi}^i}|^2}{2}\rg]\ dvol_{g_{\vec{\Phi}^i}}\quad.
\]
Assuming $x_0\in \vec{f}(\Sp^2)$, the homotopy class of $\vec{T}$ in $\pi_2(M^m,x_0)$ is the class of $\vec{f}$ and is denoted by $[\vec{T}]$.

 The subspace of elements $\vec{T}$ in
${\mathcal T}$ such that $x_0\in \vec{f}(\Sp^2)$ is denoted by ${\mathcal T}_{x_0}$.
\hfill $\Box$
\end{Dfi}

\begin{Th}
\label{th-VII.1}{\bf[Weak closure of bubble trees of weak immersions]}
Let $T_k=(\vec{f}_k,\vec{\Phi}_k^1\cdots\vec{\Phi}_k^{N_k})$ be a sequence of elements in ${\mathcal T}$ such that
\[
\limsup_{k\rightarrow +\infty}G(\vec{T}_k)=\limsup_{k\rightarrow +\infty}\sum_{i=1}^{N_k}\int_{\Sp^2}\lf[1+\frac{|D\vec{n}_{\vec{\Phi}^i_k}|^2}{2}\rg]\ dvol_{g_{\vec{\Phi}^i_k}}<+\infty\quad.
\]
Assume each $\vec{\Phi}^i_k$ is weakly conformal and that
\[
\liminf_{k\rightarrow+\infty}\sum_{i=1}^{N_k} diam\lf(\vec{\Phi}^i_k(\Sp^2)\rg)>0\quad.
\]
Then there exists a subsequence that we keep denoting $\vec{T}_k$ such that $N_k=N$ is constant and there exists a sequence of lipschitz diffeomorphisms $\Psi_k$
of $\Sp^2$ such that
\be
\label{VII.1}
\vec{f}_k\circ\Psi_k\longrightarrow \vec{u}_\infty\quad\quad\mbox{ uniformly in }C^0(\Sp^2,M^m)\quad,
\ee
where $\vec{u}_\infty\in W^{1,\infty}(\Sp^2,M^m)$, such that
\be
\label{VII.1a}
Area(\vec{f}_k(\Sp^2))\longrightarrow Area(\vec{u}_\infty(\Sp^2))\quad,
\ee
moreover for any $i=1\cdots N$ there exists $Q^i\in {\N}$ and $Q^i$ sequences of elements $f^{i,j}_k\in {\mathcal M}^+(\Sp^2)$
and for each $(i,j)$ there exist finitely many points $b^{i,j,1}\cdots b^{i,j,Q^{i,j}}$ such that
\be
\label{VII.2}
\vec{\Phi}^i_k\circ f^{i,j}_k\rightharpoonup \vec{\xi}^{i,j}_{\infty}\quad\quad\mbox{ weakly in }W^{2,2}_{loc}(\Sp^2\setminus\{b^{i,j,1}\cdots b^{i,j,Q^{i,j}}\})\quad.
\ee
The maps $\vec{\xi}^{i,j}_{\infty}$ are conformal weak immersions of ${\mathcal F}_{\Sp^2}$ and 
\be
\label{VII.3}
\lf(\vec{u}_\infty,(\vec{\xi}_\infty^{1,j})_{j=1\cdots Q^1}\cdots(\vec{\xi}_\infty^{N,j})_{j=1\cdots Q^N}\rg)\in{\mathcal T}\quad.
\ee
\hfill $\Box$
\end{Th}
\noindent{\bf Proof of Theorem~\ref{th-VII.1}.}
First of all the uniform bound on $G(\vec{T}_k)$ implies that $N_k$ is uniformly bounded as well as all the numbers $N^i_k$ associated to the underlying tree.
We can then extract a subsequence such that theses numbers are uniformly bounded and such that the associated tree $(i,j)\rightarrow J^{i,j}$ is fixed.
We can moreover find a sequence of diffeomorphism $\Psi_k$ such that, replacing $\vec{f}_k$ by $\vec{f}_k\circ\Psi_k$, the $B^i$'s, the $B^{i,j}$'s, the $\Xi^i$'s and the $b^{i,j}$'s are fixed independently of $k$ and satisfy 
\[
\forall\, i=1\cdots N \quad\quad\vec{f}_k\circ\Psi_k=\vec{\Phi}_k^i\circ\Xi^i\quad\quad\mbox{ on }\quad B^i\setminus \bigcup_{j=1}^{N^i-1}B^{i,j}\quad.
\]
Finally, since $M^m$ is compact, we can also choose the subsequence such that 
\[
\vec{f}_k\circ\Psi_k=p^{i,j}_k\rightarrow p^{i,j}_\infty\in M^m\quad\quad\mbox{ on }\quad B^{i,j}\setminus\bigcup_{i'\in \hat{J}^{i,j}} B^{i'} \quad.
\]
Now we can apply Theorem~\ref{th-I.2} to each of the sequences $\vec{\Phi}_k^i$. Modifying accordingly the diffeomorphism $\Psi_k$
in each of the $B^{i,j}\setminus\bigcup_{i'\in \hat{J}^{i,j}} B^{i'}$  and collecting all informations together provides (\ref{VII.1}), (\ref{VII.2}) and (\ref{VII.3});
Theorem~\ref{th-VII.1} is hence proved.\hfill $\Box$

\medskip

The weak closure of bubble trees of weak immersions, namely Theorem~\ref{th-VII.1}, implies in a straightforward way the following corollary once it is known that
any homotopy class $\pi_2(M,x_0)$ can be realized by an element in ${\mathcal F}_{\Sp^2}$
\begin{Co}
\label{co-VII.1}
Let $x_0\in M^m$ and let $\gamma\in\pi_2(M^m,x_0)$. Assume $\gamma\ne 0$, then
\[
\inf_{\vec{T}\in{\mathcal T}_{x_0}\ ;\ [T]=\gamma}G(\vec{T})
\]
is achieved by an element $\vec{T}_\gamma:=(\vec{f}_\gamma,\vec{\Phi}^1_\gamma\cdots\vec{\Phi}^N_\gamma)$. Moreover if we don't fix the base point
and consider now $\gamma\in \pi_2(M^m)$ and $\gamma\ne 0$, we have that
\[
\inf_{\vec{T}\in{\mathcal T}\ ;\ [T]=\gamma}G(\vec{T})
\]
is achieved by an element $\vec{G}_\gamma:=(\vec{f}_\gamma,\vec{\Phi}^1_\gamma\cdots\vec{\Phi}^N_\gamma)$.
\hfill $\Box$
\end{Co}
We will show in \cite{MoRi} that, for any $\gamma\in \pi_2(M^m)$, the $\vec{\Phi}^i_\gamma$ realize {\it area constrained smooth, possibly branched, Willmore immersions} of $\Sp^2$.

\medskip

\noindent{Proof of Corollary~\ref{co-VII.1}.} We just have to prove that for any $\gamma\in\pi_2(M^m,x_0)$ there exists an element of ${\mathcal F}_{\Sp^2}$ realizing
$\gamma$. From Theorem~\ref{th-I.1} we have a family of possibly branched conformal smooth immersions $(\vec{\Phi}^i)_{i=1\cdots N}$ which generates
$\pi_2(M^m)$ modulo the action of $\pi_1(M^m)$. In other words we can connect these immersions by tubular neighborhoods of $C^1$ paths going either from
a given base point to theses branched immersed spheres or from one of these spheres to another one in order to have a generating family
of $\pi_2(M^m)$. There is no difficulty to ``connect'' these spheres by these tubes in order to realize a branched immersion of ${\mathcal F}_{\Sp^2}$. We have then proved that
any $\gamma\in \pi_2(M^m,x_0)$ can be realized by an element of ${\mathcal F}_{\Sp^2}$. This fact combined with Theorem~\ref{th-VII.1} implies Corollary~\ref{co-VII.1}.\hfill $\Box$


\begin{thebibliography}{99}
\bibitem{BDM} D. Bartolucci, F. De Marchis, A. Malchiodi,   
\textit{ Supercritical conformal metrics on surfaces with conical singularities, } Int. Math. Res. Not.,  (2011), Num. 24, 5625--5643.
\bibitem{BR1} Y.  Bernard, T. Rivi\`ere,  \textit{Local Palais Smale Sequences for the Willmore Functional} ,  Comm. Anal. Geom.,  (2011),  Num. 3, 563--599.
\bibitem{BR2} Y.  Bernard, T. Rivi\`ere,  \textit{Asymptotic Analysis of Branched Willmore Surfaces},  preprint  (2011).
\bibitem{BR3} Y.  Bernard, T. Rivi\`ere,  \textit{Energy Quantization for Willmore Surfaces and Applications}, preprint (2011).
\bibitem{DoC} M. do Carmo ,  
{\it Riemannian Geometry} Mathematics: Theory \& applications. Birkauser Boston, Inc., Boston, MA, (1992).
\bibitem{Ge} Y. Ge, \textit{Estimations of the best constant involving the $L^2$ norm in Wente's inequality
and compact $H-$Surfaces in Euclidian space,}  C.O.C.V., (1998),  Num.  3, 263-300.
\bibitem{GT}  D. Gilbarg, N. Trudinger, \textit{ Elliptic partial differential equations of second order. Reprint of the 1998 edition. } Classics in Mathematics. Springer-Verlag, Berlin, (2001).
\bibitem{He} F.  H\'elein, \textit{Harmonic Maps, Conservation Laws, and Moving Frames}, Cambridge Tracts in Mathematics, 150. Cambridge University Press (2002).
\bibitem{IT} Y. Imayoshi, M. Taniguchi, \textit{An Introduction to Teichm\"uller Spaces} Springer (1992).
\bibitem{Jo} J. Jost,  
{\it Compact Riemann Surfaces. An introduction to contemporary mathematics. Third edition. } Universitext. Springer-Verlag, Berlin, (2006).
\bibitem{KMS} E. Kuwert, A. Mondino, J. Schygulla 
{\it Existence of immersed spheres minimizing curvature functionals
in compact 3-manifolds,}   arXiv:1111.4893, preprint (2011).
\bibitem{KS} E. Kuwert, R. Sch\"atzle,
{\it Removability of isolated singularities of Willmore surfaces,} Annals of Math. Vol. 160, Num. 1, (2004), 315--357. 
\bibitem{McOw} R. C. McOwen,
{\it Prescribed curvature and singularities of conformal metrics on Riemann surfaces}, Journal of Math. Anal. and Appl., Vol.177,  (1993), 287--298.
\bibitem{Mon1} A. Mondino, 
{\it Some results about the existence of critical points for the Willmore functional,} Math. Zeit., Vol. $266$, Num. $3$, (2010), 583--622.
\bibitem{Mon2} A. Mondino, 
{\it The conformal Willmore Functional: a perturbative approach,}  J. Geom. Anal. , Vol. 23,  (2013), 764--811.
\bibitem{MoRi} A. Mondino, T. Rivi\`ere, \textit{Willmore spheres in compact Riemannian manifolds},  Advances in Math., Vol. 232, Num.1, (2013), 608--676.
\bibitem{MS}  A. Mondino, J. Schygulla, 
{\it Existence of immersed spheres minimizing curvature functionals in non-compact 3-manifolds,}   arXiv:1201.2165, preprint (2012).
 \bibitem{MS}  S. M\"uller, V.  \v{S}ver\'ak, \textit{On surfaces of finite total curvature,}  J. Diff. Geom., Vol. 42 , Num. 2, (1995), 229--258.
\bibitem{Ri1}  T. Rivi\`ere, \textit{Conformally Invariant  Variational Problems} Cours joint de l'Institut Henri Poincar\'e - Paris XII Creteil, Novembre 2010,  arXiv:1206.2116. 
\bibitem{Ri2} T. Rivi\`ere, \textit{Analysis aspects of Willmore surfaces}, Inventiones Math., 174 (2008), Num.1, 1-45.
\bibitem{Ri3} T. Rivi\`ere, \textit{Variational Principles for immersed Surfaces with $L^2$-bounded Second Fundamental Form} arXiv:1007.2997 (2010), to appear on Crelle's Journal. 
\bibitem{Ri4} T.  Rivi\`ere, \textit{Lipschitz conformal immersions from degenerating Riemann surfaces with $L^2-$ bounded
second fundamental forms},  Adv. Calc. Var., Vol. 6, Num. 1,  (2013), 1–31 .
\bibitem{SaU} J. Sacks and K. Uhlenbeck, \textit{The existence of minimal immersions of 2-spheres}, Ann. of Math., Vol. 113,  (1981), 1--24.
\bibitem{Si} L. Simon,\textit{Existence of surfaces minimizing the Willmore functional}, Comm. Anal. Geom., Vol. 1, (1993), 281-326
 \bibitem{Top}  P. Topping, \textit{The optimal constant in Wente's $L^\infty$ estimate}, Comm. Math. Helv., Vol. 72, (1997), 316-328.
\bibitem{Tro} M. Trojanov, 
 \textit{Prescribing curvature on compact surfaces with conical singularities}, Trans A.M.S., Vol. 324, Num. 2, (1991), 793-821.

 \end{thebibliography}
\end{document}